\documentclass{article}
\usepackage{amsthm,amsmath,amsfonts,amssymb}
\title{Non-uniqueness of the Leray-Hopf solutions in the hyperbolic setting.}
 
\author{Chi Hin Chan\footnote{Institute for Mathematics and its Applications,
University of Minnesota,
207 Church Street SE,
Minneapolis, MN 55455-0134,
\texttt{chana002@ima.umn.edu}
} \and Magdalena Czubak\footnote{Department of Mathematics, University of Toronto, 40 St. George St.
Toronto, Ontario, M5S 2E4, Canada, \texttt{czubak@math.toronto.edu}}}

\newlength{\hchng}
\newlength{\vchng}
\newtheorem*{nonumthm}{Theorem}

\newtheorem{thm}{Theorem}[section]
\newtheorem{prop}[thm]{Proposition}
\newtheorem{cor}[thm]{Corollary}
\newtheorem{lemma}[thm]{Lemma}
\newtheorem{defn}[thm]{Definition}

\newtheorem{preremark}[thm]{Remark}
\newenvironment{remark}{\begin{preremark}\rm}{\medskip \end{preremark}}

\numberwithin{equation}{section}

\newcommand{\norm}[1]{\left\Vert#1\right\Vert}
\newcommand{\abs}[1]{\left\vert#1\right\vert}

\newcommand{\R}{\mathbb R}

\newcommand{\grad} {\nabla}

\DeclareMathOperator*{\osc}{osc}

\DeclareMathOperator{\dv}{div}
\DeclareMathOperator{\Def}{Def}
\DeclareMathOperator{\Ric}{Ric}
\DeclareMathOperator{\N}{N-S}

\def\H{\mathbb H^{2}(-a^{2})}
\def\be{\begin{equation}}
\def\ee{\end{equation}}

\newcommand{\NS}[1]{\N_{#1}}

\begin{document}
\maketitle
\begin{abstract}
We consider the Navier-Stokes equation on $\H$, the two dimensional
hyperbolic space with constant sectional curvature
$-a^{2}$.  We prove an ill-posedness result in the sense that the uniqueness of the
Leray-Hopf weak solutions to the Navier-Stokes equation breaks down on $\mathbb{H}^{2}(-a^{2})$.  We also obtain a corresponding result on a more general negatively curved manifold for a modified geometric version of the Navier-Stokes equation. Finally, as a corollary we also show a lack of the Liouville theorem in the hyperbolic setting both in two and three dimensions.    
\end{abstract}

\section{Introduction}
We investigate the impact the geometry of the underlying space has on the Leray-Hopf solutions to the Navier-Stokes equation.  More precisely, we consider the Navier-Stokes equation on negatively curved manifolds and present how the negative scalar curvature causes the break down of the uniqueness of the Leray-Hopf solutions.   
\newline\indent
Before we state the main
results, we survey some necessary historical background from both geometric analysis, and the regularity theory for the Navier-Stokes
equation in the usual Euclidean setting.
\subsection{Regularity theory for the Navier-Stokes equation on $\R^{n}$}
The Navier-Stokes equation on the Euclidean space $\mathbb{R}^{n}$
is given by
\begin{equation}\tag{$\NS{\R^{n}}$}
\begin{split}
\partial_{t}u -\Delta u + u\cdot \nabla u + \nabla P & = 0 , \\
\dv u& = 0 . 
\end{split}
\end{equation}
Long time ago, for the dimensions $n=2$ and $n=3$, 
Leray~\cite{Leray} and Hopf~\cite{Hopf} established the existence of global
weak solutions $u\in L^{\infty}(0, \infty ; L^{2}(\mathbb{R}^{n}) ) \cap
L^{2}(0,\infty ;\dot H^{1}(\mathbb{R}^{n}))$.  Due to their work, we now have the following
general existence result, which historically served as the foundation for further
works in the regularity theory for $\NS{\R^{n}}$. 
\begin{nonumthm}
[Leray-Hopf weak solutions \cite{Leray, Hopf}] Given any initial datum $u_{0} \in L^{2}(\mathbb{R}^{n})$,
there exists at least one $\R^{n}$-valued function $u \in L^{\infty}(0, \infty ;
L^{2}(\mathbb{R}^{n}) ) \cap L^{2}(0,\infty ; H^{1}(\mathbb{R}^{n}))$ which
satisfies the following properties
\begin{itemize}
\item For any $\phi = (\phi_{1}, ..., \phi_{n}) \in C_{c}^{\infty}((0,\infty )\times
\mathbb{R}^{n})$ with $\dv \phi = 0$, we have
\begin{equation*}
\int_{0}^{\infty}\int_{\mathbb{R}^{n}} -u\cdot \partial_{t}\phi + \sum_{i,j}
(\partial_{j}\phi_{i}) (\partial_{j}u_{i}) -
\sum_{i,j}(\partial_{j}\phi_{i})(u_{i}u_{j}) dx dt = 0.
\end{equation*}
\item For every $t \geq 0$, $u$ satisfies the following global energy inequality
\begin{equation*}
\int_{\mathbb{R}^{n}} |u(t,x)|^{2} dx + 2 \int_{0}^{t} \int_{\mathbb{R}^{n}}|\nabla
u|^{2} dx ds \leq  \int_{\mathbb{R}^{n}} |u_{0}|^{2} dx .
\end{equation*}
\item $u(0,\cdot )$ coincides with the initial datum $u_{0}$ in the sense that 
\begin{equation}
lim_{t\rightarrow 0^{+}} \|u(t,\cdot ) - u_{0}\|_{L^{2}(\mathbb{R}^{n})} = 0.
\end{equation}
\end{itemize}
\end{nonumthm}
Now that we have existence of the Leray-Hopf solutions for $\NS{\R^{2}}$ and $\NS{\R^{3}},$ we proceed to address the question of regularity.  The regularity of Leray-Hopf solutions on $\mathbb{R}^{2}$ greatly differs from the corresponding regularity problem for Navier-Stokes equation on
$\mathbb{R}^{3}$. Indeed, the smoothness and uniqueness of Leray-Hopf
solutions for $\NS{\R^{2}}$ was established in the
work of Leray (see for instance \cite{MichaelTaylor}).  As a sharp contrast, 
the regularity and uniqueness of solutions to the $\NS{\R^{3}}$ equation is a long standing open problem although due to the concentrated efforts by generations of PDE specialists there has been a significant progress in this area.
\newline\indent
Because of the limitation of space, we only mention some typical
regularity criteria for Leray-Hopf solutions to $\NS{\R^{3}}$. 
We also note that one of the goals of this discussion is to illustrate why there is more focus on the question of regularity than that of the uniqueness.
\newline\indent
Now, the first significant effort to break the silence since the fundamental work of
Leray and Hopf, was made in 1960's
through the efforts of 
Prodi~\cite{Prodi}, Serrin~\cite{Serrin}, and Ladyzhenskaya \cite{Ladyzhenskaya} leading to the
following regularity and uniqueness result (for more historical remarks see for
instance \cite{LinftyL3}). 
\begin{thm}\label{heatequationpertubation}
[Prodi, Serrin, Ladyzhenskaya] Let $u \in L^{\infty}(0, T ;
L^{2}(\mathbb{R}^{3}) ) \cap L^{2}(0, T ; \dot H^{1}(\mathbb{R}^{3}))$ to be a
Leray-Hopf weak solution to $\NS{\R^{3}}$, which
satisfies the additional condition that $u \in L^{p}(0,\infty ;
L^{q}(\mathbb{R}^{3}))$, for some $p, q$ satisfying $\frac{2}{p} + \frac{3}{q} = 1,$
with $q > 3$. Then, $u$ is smooth on $(0,T]\times \mathbb{R}^{3}$ and $u$ is
uniquely determined in the following sense
\begin{itemize}
\item suppose $v \in L^{\infty}(0, T ; L^{2}(\mathbb{R}^{3}) ) \cap L^{2}(0, T ;
H^{1}(\mathbb{R}^{3}))$ is another Leray-Hopf weak solution such that $u(0, \cdot ) = v(0,\cdot )$. Then, it follows that
$u=v$ on $(0,T]\times \mathbb{R}^{3}$. 
\end{itemize}
\end{thm}
Here, let us briefly mention why the case of $q=3$ was not included in Theorem
\ref{heatequationpertubation}. Indeed, it is well known that a solution $\theta :
[0,T)\times \mathbb{R}^{3}\rightarrow \mathbb{R}$ to the \emph{heat equation}
arising from any initial datum $\theta_{0} \in L^{3}(\mathbb{R}^{3})$ satisfies the
following estimate for any pair of indices $p,q$ with
$\frac{2}{p}+\frac{3}{q} = 1$ \emph{and} $q > 3$ (see \cite[Appendix]{LinftyL3}).

\begin{equation}
\|\theta \|_{L^{p}(0,T; L^{q}(\mathbb{R}^{3}))} \leq C(q)
\|u_{0}\|_{L^{3}(\mathbb{R}^{3})} ,
\end{equation}
where $C_{q}$ depends only on $q$. So, in some sense,
the extra condition as imposed on the Leray-Hopf solution $u$ in
Theorem~\ref{heatequationpertubation} ensures that the qualitative behavior of the
Leray-Hopf solution $u$ would be a slight perturbation of
solutions of the heat equation.  Another explanation for the relatively simple
status of Theorem~\ref{heatequationpertubation} is that the $L^{p}_{t}L^{q}_{x}$ norm of
the solution $u$ under the integral condition as promised in
Theorem~\ref{heatequationpertubation} shrinks to zero under the natural scaling
$u_{\epsilon}(t,x) = \epsilon u(\epsilon^{2}t , \epsilon x)$ as $\epsilon
\rightarrow 0$. However, this is no longer valid in the borderline case of
$L^{\infty}(0,T; L^{3}(\mathbb{R}^{3}))$. This partially explains the long delay in
the settlement of this exceptional case of $u \in L^{\infty}(0,T;
L^{3}(\mathbb{R}^{3}))$, which was finally established in the recent work
of Escauriaza, Seregin, and \v{S}ver\'ak \cite{LinftyL3}.
\newline\indent
Before we close our discussion let us mention that one of
the working principles in the regularity theory of Navier-Stokes equations on
$\mathbb{R}^{3}$ is (more or less) to reduce the situation under consideration (say
$u \in L^{\infty}(L^{3})$ in the case of \cite{LinftyL3})
to the regime which is covered by Theorem~\ref{heatequationpertubation}.  Once this
can be achieved, then the uniqueness of the solution would come
for free, due to the uniqueness claim in  Theorem~\ref{heatequationpertubation}.
This explains to some extent the fact that regularity issue is more of a focus than the
uniqueness issue in the regularity theory for Navier-Stokes equations in the
$\mathbb{R}^{3}$ setting. However, as is well-known, the weak formulation for
Leray-Hopf weak solutions to the Navier-Stokes equation on $\mathbb{R}^{3}$ only
gives the natural bound 
$u \in L^{p}(0,\infty ; L^{q}(\mathbb{R}^{3}))$, with indices $p,q$ satisfying
$\frac{2}{p} + \frac{3}{q} = \frac{3}{2}$. One readily sees that there is a
significant gap between the natural bound offered by the weak formulation and the
condition required by Theorem~\ref{heatequationpertubation}, and it is unclear how
to make a link between them.   See again the introduction of \cite{LinftyL3} for a
discussion about this point, and for further developments, we refer our
readers to a piece of recent work  by Vasseur \cite{Higherderivatives}. 

\subsection{Navier Stokes equation on a Riemannian manifold}
Historically speaking, the correct form of the Navier-Stokes equations in the
Riemannian manifold setting was first obtained by Ebin and Marsden \cite{EbinMarsden}.  They considered compact Riemannian, oriented, $n$-dimensional manifolds both with and without boundary.  Moreover, they remark that the derivation of the correct equations assumes that the manifold in question is Einstein, i.e.,$\Ric=\lambda g$, for some constant 
$\lambda$ where $\Ric$ is the Ricci tensor and $g$ is the Riemannian metric $g$.  We note, this is in particular true of space forms, where $\Ric=(n-1)K_{M}g$ (see Section \ref{diffgeom} below).  
\newline\indent
According to \cite{EbinMarsden} the ordinary Laplacian should be replaced by the following
operator in the formulation of the Navier-Stokes equation on a Riemannian manifold  
\begin{equation}\label{canconicalLaplacian}
L = 2 Def^{*}Def = \overline{\nabla}^{*} \overline{\nabla} + dd^{*} - \Ric = (dd^{*}
+ d^{*} d) + dd^{*} -2 \Ric ,
\end{equation}
where  $Def$ and $Def^{*}$ are the stress tensor and its adjoint respectively,
$\overline{\nabla}$ stands for the induced Levi-Civita connection on the cotangent
bundle $T^{*}M$, $\overline{\nabla}^{*}
\overline{\nabla}$ stands for the Bochner Laplacian, with $\overline{\nabla}^{*}$ to
be the adjoint operator associated to $ \overline{\nabla}$, $(dd^{*} +
d^{*} d)= -\Delta$
stands for the Hodge Laplacian with $d^{*}$ to be the formal adjoint of the exterior
differential operator $d$, and $\Ric$ is the Ricci operator (see Sections \ref{diffgeom} and \ref{hodge} for definitions and \cite{DindosMitrea, MichaelTaylor} for a further discussion of the deformation tensor). We first remark that the operator $L$
as given in expression (\ref{canconicalLaplacian}) is an operator sending sections
of $T^{*}M$ into sections of $T^{*}M$. This means that, the Navier-Stokes equations
on a Riemannian manifold $M$ is formulated in terms of sections of $T^{*}M$ instead
of vector fields on $M$.  
\newline\indent
As a result, the usual convection term $\nabla_{u}u$ in terms of vector fields also has to be rewritten.  There is a natural correspondence between vector fields and $1$-forms (see Section 2.1), which produces the term  $\overline{\nabla}_{U}U^{\ast}$, where $U$ is the unique vector field corresponding to a $1$-form $U^{\ast}$. 
\newline\indent
In summary, we regard the solutions to the Navier-Stokes
equations on a general Riemannian manifold $M$ to be differential $1$-forms $U^\ast
\in C^{\infty}(M ; T^{*}M) $ satisfying the following differential equation
\begin{align}
&\partial_{t}U^{\ast} + L (U^{\ast}) + \overline{\nabla}_{U}U^{\ast} + dP = 0, \\
&d^{\ast}U^{\ast}=0.
\end{align}
where $P$ is a scalar function on $M$.  Using the fact that $U^{\ast}$ is divergence free we can further rewrite the equations as follows
%
%

\begin{equation}\tag{$\NS{M}$}
\begin{split}
\partial_{t}U^{*} - \Delta U^{*}  + \overline{\nabla}_{U}U^{*} - 2\Ric(U^{\ast})+ dP &= 0,\\
d^{\ast}U^{\ast}&=0.
\end{split}
\ee
which is the main equation that we study in this article.
\newline\indent
A less natural equation to consider is one without the Ricci operator.  We refer to it as the \emph{modified} Navier-Stokes equations on $M$ and record it here
\begin{equation}\label{MNS}
\begin{split}
\partial_{t}U^{*} - \Delta U^{*}  + \overline{\nabla}_{U}U^{*} + dP &= 0.\\
d^{\ast}U^{\ast}&=0
\end{split}
\end{equation}
It is less natural from the point of view of the derivation of the Navier-Stokes equations performed in \cite{EbinMarsden}.  We consider it in this paper, because we would like to present how a more general manifold than just a space form can influence the behavior of solutions (we explain this more below).  
\newline\indent
Since $\Def^{\ast}\Def U^{\ast}$ plays now the role of the dissipation, the global energy inequality becomes
\begin{equation}\label{energyM}
\int_{M} \abs{U^{\ast}}^{2}(t,x)  + 2 \int_{0}^{t} \int_{M}      \overline{g} (\Def U^{\ast} , \Def U^{\ast}  )       ds \leq  \int_{M} |u_{0}|^{2} ,
\end{equation}
where $\overline{g}( \cdot , \cdot )$ stands for the inner product structure on the bundle $T^{*}M \otimes T^{*}M$ \emph{induced by the Riemannian metric $g(\cdot, \cdot )$ on $M$} (see Section \ref{diffgeom}).
\newline\indent
We now mention some of the previous results on a Riemannian manifold (for more see \cite{DindosMitrea} and references therein).
Priebe \cite{Priebe} appears to be the first one to follow \cite{EbinMarsden} and use the correct version of the equations $\NS{M}$ instead of \eqref{MNS}.  \cite{Priebe} also assumes compactness of $M$ and works on manifolds with boundary.  Dindos and Mitrea \cite{DindosMitrea} consider the linearized version of the stationary Navier-Stokes equations on
a subdomain of a compact Riemannian manifold.
In fact, we have not been able to find any results for non-compact manifolds except for the result of Q.S. Zhang \cite{QSZhang}.  In \cite{QSZhang} the author shows the ill-posedness of the \emph{weak solution with finite $L^{2}$ norm} on a connected sum of two copies of $\R^{3}$.  Hence the topology of the underlying manifold seems to play a role.  In this paper, we take a geometric point of view and also consider the dissipation term, which involves careful computations.  
\newline\indent
We are now ready to state our main results.
\subsection{Statements of the results and discussion of the proofs}  
\begin{thm}[Non-uniqueness of $\NS{\H}$]\label{mainthm}
Let $a>0$. Then there exist non-unique Leray-Hopf solutions to $\NS{\H}$.
\end{thm}
\begin{remark}
The consequence of Theorem \ref{mainthm} is that unlike in the Euclidean setting, the notion of the Leray-Hopf solutions might not be the proper foundation for the study of  solutions to the Navier-Stokes equations on the space form with negative sectional curvature $-a^{2}$ in dimension two.  The question of what happens to the strength of the framework of the Leray-Hopf solutions on $H^{3}(-a^{2})$ is open.  The lack of a definite answer at this point is mainly due to the specific form of the techniques we use in this paper (see remarks below).  It is an interesting question to see if perhaps the techniques could eventually be extended/modified to give some insight into the corresponding question in 3 dimensions.
\end{remark}
\begin{cor}[Lack of the Liouville theorem for space forms]\label{cor1}
Let  $n\geq 2$, and $a>0$ then there exist nontrivial bounded solutions of $\NS{H^{n}(-a^{2})}$.   
\end{cor}
\begin{remark}
The proof of Corollary \ref{cor1} and Corollary \ref{cor2} below follows trivially from the proofs of their theorems.  Moreover, it does not require any of the delicate estimates developed in this paper.  As such it is just a by-product of the main results and we only include it here for completeness, and because of the general importance the Liouville theorems play in the subject of the Navier-Stokes equations.   See Section \ref{corproofs} for motivation and some background. 
\end{remark}
If one decides to omit the Ricci term from the equation, we can also have a non-uniqueness result on a more general negatively curved Riemannian manifold than just $\H$. 
\begin{thm}\label{thm2}
Let  $a, b>0$ be such that $\frac 12 b<a\leq b$, and let $M$ be a simply connected, complete $2$-dimensional Riemannian manifold with sectional curvature satisfying $-b^{2} \leq K_{M} \leq -a^{2}$.  Then there exists non-unique Leray-Hopf solutions to the modified Navier-Stokes equation \eqref{MNS}.
\end{thm}
\begin{cor}[Lack of the Liouville theorem in the hyperbolic setting]\label{cor2}
Let  $n\geq 2$, and $b\geq a>0$ and let $M$ be a simply connected, complete $n$-dimensional Riemannian manifold with sectional curvature satisfying $-b^{2} \leq K_{M} \leq -a^{2}$.  Then there exist nontrivial bounded solutions of \eqref{MNS}.   
\end{cor}
\begin{remark}
Note, the lower bound $\frac 12 b<a$ is no longer required in the corollary.  See the discussion below for why the lower bound is present in Theorem \ref{thm2}, which also explains why we do not need it in Corollary \ref{cor2}.
\end{remark}
The above results are based on the abundance of the \emph{nontrival bounded harmonic functions} in the hyperbolic setting.  Such abundance is ensured by the works of Anderson \cite{Anderson} and Sullivan \cite{Sullivan}.   Our idea of trying to benefit from them was inspired by a remark of Tsai \cite[Remark 5.4]{Tsai}.  \cite{Tsai} eliminates a possibility of self-similar solutions to $\NS{\R^{3}}$ (which merely would satisfy the local energy inequality) by showing that existence of the self-similar solutions is equivalent to solving a certain stationary system.  Without assuming enough decay, one could construct nontrival solutions of the system in question in the form of $U = \nabla F$, and $P = -\frac{1}{2} |U|^{2}-a y\cdot U,$ where $F$ is a harmonic function on $\mathbb{R}^{3}$ and $a>0$.  In our case, due to \cite{Anderson} and \cite{Sullivan} we have a plethora of 
 \emph{nontrival bounded harmonic functions}, which gives us a basis for this article.
\newline\indent
The solution pairs $(U^{\ast} , P)$ we consider have the following form
\begin{equation}\label{strategery}
\begin{split}
U^{*} &= \psi(t) dF ,\\
 P &= -\partial_{t}\psi(t) F - \frac{1}{2} |dF|^{2} + 2a^{2}F , 
\end{split}
\end{equation}
where  $\psi(t) = \exp(-\frac{At}{2})$ for some $A \geq 2a^{2}$, and $F$ is a nontrival bounded harmonic function on $\mathbb{H}^{2}(-a^{2})$.
Verifying that $(U^{\ast}, P)$ solves $\NS{\H}$ is simple when we use Hodge theory (see Sections \ref{hodge} and \ref{proofs}) and Lemma \ref{convection}.
\footnote{In fact, taking solutions of the form $\psi(t)\grad F$ seems to be a well known convention, and we just happened to learn about it from \cite{Tsai}.  But also see \cite{QSZhang}.}
\newline\indent
Before we proceed any further, we remark here that the \emph{differential geometric work} in \cite{Anderson, Sullivan, AndersonSchoen} ensures the existence of nontrival bounded harmonic functions on a more general \emph{negatively curved Riemannian manifold} with suitable lower and upper bounds imposed on the sectional curvature. On the other hand, the existence of nontrival bounded harmonic function on $\mathbb{H}^{n}(-1)$ is an old classical result obtained through an integral representation formula with an explicit Poisson kernel on the Poincare ball model for the space form  $\mathbb{H}^{n}(-1)$ (for more details see the work of Hua \cite{Hua}).
However, such classical approach relies heavily on the explicit formula of the Poisson kernel derived from the group of isometries of the space form $\mathbb{H}^{n}(-1)$.  It seems that, as compared with the differential geometric approach of \cite{Anderson, Sullivan, AndersonSchoen}, such classical approach does not reveal the role played by the \emph{negative sectional curvature} of the hyperbolic manifold in the existence of nontrival bounded harmonic functions on $\mathbb{H}^{n}(-1)$.
\newline\indent
The last remark may explain why the proper generalization of the above mentioned classical result to the more general setting of negatively curved Riemannian manifolds was only established in the more recent works \cite{Anderson, Sullivan, AndersonSchoen}. Since we intend to show not only the break down of the uniqueness of Leray-Hopf solutions in the hyperbolic space setting, but more importantly the \emph{decisive role played by the negative sectional curvature of a hyperbolic manifold in causing such a breakdown}, we will \emph{unconditionally} choose the differential geometric framework as established in 
\cite{Anderson, Sullivan, AndersonSchoen} as the basic ground in this paper.
\newline\indent
Moreover, since the differential geometric machinery as demonstrated in \cite{Anderson}, \cite{Sullivan}, and \cite{AndersonSchoen} is designed to establish the existence of bounded nonconstant harmonic function on a general negatively curved Riemannian manifold which \emph{lacks the homogeneity property enjoyed by the space form $\mathbb{H}^{2}(-a^{2})$}, the only best way to justify the use of such differential geometric machinery in our paper is to cast our theorems, lemmas, propositions in the most general setting of a negatively curved Riemannian manifold at the \emph{starting point} of the paper. However, we slowly narrow down our setting by imposing further restrictions on our results whenever such restrictions are needed in proving the finite integral of a certain function or in handling the extra $\Ric$ term in the formulation of the Navier-Stokes equations. 
\newline\indent
As stated in Theorem \ref{thm2}, our non-uniqueness result also holds for a more general negatively curved Riemannian manifold with the lower bound $-b^{2}$ and the upper bound $-a^{2}$ of the sectional curvature satisfying $0 < \frac{b}{2} < a \leq b$, \emph{provided if the extra $\Ric$ term in the Navier-Stokes equation is dropped}. Indeed, the final restriction to the space form $\mathbb{H}^{2}(-a^{2})$ is required only because of the presence of $-2 \Ric$ in the formulation of the Navier-Stokes equation $\NS {M}$. 
\newline\indent
Now, we explain our strategy in establishing the finite energy and the finite dissipation of the time dependent velocity field $U^{*} = \psi(t) dF$ as given in equation \eqref{strategery}. We start our discussion by saying that our exposition is based on the material in the second chapter of the book \cite{redbook} by R. Schoen and S.-T. Yau.
In the first section of the second chapter of \cite{redbook}, one sees that, with the prescribed function
$\phi \in C^{1}(S(\infty))$ given on the geometric boundary $S(\infty)$ (see Section \ref{balls}) attached to the $n$-dim complete, simply connected Riemannian manifold $M$ with sectional curvature satisfying $-b^{2} \leq K_{M} \leq -a^{2} < 0$, the bounded harmonic function $F$ on $M$ which satisfies the Dirichlet boundary condition $F|_{S(\infty)}=\phi$ is sought after by means of creating two barrier functions $\overline{\phi} - \alpha e^{-\delta \rho}$ and $\overline{\phi} + \alpha e^{-\delta \rho}$, which serve as the lower bound and upper bound for $F$ and where $\rho$ stands for the distance function on $M$ from a selected base point $O$ in $M$ (also see Section \ref{harmonic} and \ref{moreprelim}).  This is done in the spirit of the classical Perron's method. But such an application requires the subharmonicity of $\overline{\phi} - \alpha e^{-\delta \rho}$ and the superharmonicity of $\overline{\phi} + \alpha e^{-\delta \rho}$, whose validity critically depends on the
  following two facts (for details, see \cite{redbook}) 
\begin{itemize}
\item \emph{Laplace comparison theorem}:  If $K_{M} \leq -a^{2}$, then, the Laplacian of the distance function $\rho$ (from a selected base point $O$ in $M$) satisfies $\Delta \rho \geq (n-1)a \coth (a\rho) \geq (n-1)a$. 
\item the \emph{smooth} function $\overline{\phi}$ is constructed in a specific way so that we have $\overline{\phi}|_{S(\infty)} = \phi$ and that the oscillation of $\overline{\phi}$ over any geodesic ball $B_{x}(1)$ in $M$ has exponential decay of order $e^{-a\rho (x)}$, for any $x\in M$.
\end{itemize}
Due to the above two facts, it can be deduced that the choice of the $\delta > 0$, which ensures $\Delta [\overline{\phi} - \alpha e^{-\delta \rho}] \geq 0$ and $\Delta [\overline{\phi} + \alpha e^{-\delta \rho}] \leq 0$ (and hence the success of the Perron's method), has to satisfy the constraint $\delta < a$ (see Section \ref{constants}). 
\newline\indent
Based on what we learn from the above construction of the bounded nontrival harmonic function $F$ on $M$ we employ, in Section 3 the gradient estimate for harmonic functions due to S.-T. Yau \cite{Yau75} to show that the decay rate for $|\nabla F|(x)$ as $\rho(x)$ approaches infinity is at least of the order $e^{-\delta \rho (x)}$, for any $\delta < a$. That is we have
$$|\nabla F| \leq C(a, \delta )\|\phi'\|_{S(\infty)} e^{-\delta \rho},$$ on $M$.  Here, we want to mention that, with the hindsight from one of the two Harnack's inequalities as established in the second chapter of \cite{redbook}, one can argue that such an exponential decay for the gradient of our bounded harmonic function is more or less \emph{expected} and may not be surprising. We believe that such an exponential decay could be more or less well known to researchers working in geometric analysis. But in any case, we give a clean and simple proof of it in Section 3.
\newline\indent
Next we note that the exponential decay $|\nabla F| \leq C(a, \delta )\|\phi'\|_{S(\infty)} e^{-\delta \rho}$ not only gives, in the special case of the two dimensional space form $ \mathbb{H}^{2}(-a^{2}) $, the $L^{2}$
-finite property of $|\nabla F|$ on $ \mathbb{H}^{2}(-a^{2}) $ (and hence the finite energy property of the velocity field $U^{*} = \psi(t) dF$ as given in \eqref{strategery}), \emph{but also demonstrates the limitation which prevents us to draw the same $L^{2}$-finite property of $|\nabla F|$ in the setting of the three-dimensional space form $ \mathbb{H}^{3}(-a^{2}) $ }.  This limitation mainly comes from the fact that $\osc_{B_{x}(1)}\overline{\phi}$ only has exponential decay of order $e^{-a \rho(x)}$, which prevents us from choosing a $\delta > 0$ larger than $a$ in the Perron's method;  yet the \emph{growth rate} of the volume form on $ \mathbb{H}^{3}(-a^{2}) $ is exactly of the order $\frac{1}{a^{2}}\sinh^{2}(a\rho)$. For the first time, we encounter an obstacle which forces us to restrict our theory only to the case of $2$ dimensional negatively curved Riemannian manifold $M$ with $-b^{2} \leq K_{M} \leq -a^{2} < 0$.
\newline\indent
We observe that, up to this point, the lower bound condition $K_{M} \geq -b^{2}$ has not been involved in the big picture yet.  However one does eventually have to pay a special attention to the relative largeness of $b$ when compared with $a$ since the lower bound $K_{M} \geq -b^{2}$ of the sectional curvature determines the \emph{growth rate} of the volume form of $M$ via the comparison theorem for Jacobi fields in differential geometry.  More precisely, with the condition  $K_{M} \geq -b^{2}$ imposed, the growth rate of the volume form of the $2$-dim negatively curved manifold $M$ is at most $\frac{1}{b} \sinh(b \rho)$. Yet, again, the decay rate of $|\nabla F|$ is of the order $e^{-\delta a}$, with any $\delta <a$. As a result, the survival of the property $\|\nabla F\|_{L^{2}(M)} <\infty$ critically depends on the competition between the decay rate $e^{-\delta a}$ of $|\nabla F|$ and the (possible) worst growth rate $\frac{1}{b} \sinh(b \rho)$ of the volume form of $M$. 
 This fully explains the need to further restrict our setting by imposing the condition $0< \frac{b}{2} < a \leq b$, so that \emph{the parameter $\delta$ can fit within the range of $\frac{b}{2} < \delta < a$, which is enough to ensure the survival of the $L^{2}$-finite property of $|\nabla F|$ on the $2$-dim negatively curved manifold $M$}.
\newline\indent
Once the $L^{2}$-finite property of $|\nabla F|$ is established for $2$-dim Riemannian manifold satisfying $-b^{2} \leq K_{M} \leq -a^{2}$ \emph{and} $0<\frac{b}{2} < a \leq b $, we proceed to show the finite dissipation of $U^{*} = \psi(t) dF$ \emph{under the same setting} in Sections 4 and 5 of our paper, which are the most delicate parts of our work. In regard to this, our basic idea lies in the structure of the following important formula in differential geometry \cite{redbook}
\begin{equation}\label{crucialformulaIntroduction}
\Delta(|\nabla F|^{2})(x) = 2 \overline{g} (\overline{\nabla}(\nabla F) , \overline{\nabla}(\nabla F) )(x) + 2 \sum \partial_{i}F(x) \partial_{i}(\Delta F)(x) + 2 Ric (\nabla F, \nabla F)(x).
\end{equation}
The formula is obtained by performing a calculation with respect to the normal geodesic coordinates about the selected point $x$ in our 2-dim negatively curved Riemannian manifold $M$.
Since $\Delta(|\nabla F|^{2}) = div \{\nabla (|\nabla F|^{2}     )\}$, if \emph{we can show that the $L^{1}$-norm of $|\nabla |\nabla F|^{2}|$ is finite}, then, we will immediately have
$\int_{M} \Delta(|\nabla F|^{2}) =0$. Hence, it follows from the above formula that (see Proposition 5.1 of Section 5 for the technical details) 
\begin{equation}\label{consequencebackbone}
 \int_{M} \overline{g} (\overline{\nabla}(\nabla F) , \overline{\nabla}(\nabla F))= - \int_{M} Ric (\nabla F , \nabla F) \leq b^{2} \int_{M} |\nabla F|^{2} , 
\end{equation}
which gives the finiteness of the dissipation term $\int_{M} \overline{g} (\overline{\nabla}U , \overline{\nabla}U) = [\psi(t)]^{2}  \int_{M} \overline{g} (\overline{\nabla}(\nabla F) , \overline{\nabla}(\nabla F))$) as required in the Leray-Hopf formulation.
\newline\indent
Next, the required $L^{1}$-finite property of $\abs{\nabla |\nabla F|^{2}}$ is established with the assistance of a covering Lemma 4.1.  
Due to the fact that $\Ric (\nabla F, \nabla F) \geq -b^{2} |\nabla F|^{2}$ and that $\Delta F = 0$, formula \eqref{crucialformulaIntroduction} ensures the subharmonicity of $|\nabla F|^{2} + A e^{-2\delta \rho}$ on $\{x \in M : \rho (x) > \frac{R(\delta)}{b}\}$ (with $\frac{b}{2} < \delta < a$), for some sufficiently large $A > 0$ and $R(\delta ) > 0$, both dependent on $\delta$, which in turns allows us to obtain (see the proof of Proposition 4.4 for the technical details)
\begin{itemize}
\item  the integral of $|\nabla\{ |\nabla F|^{2} + A e^{-2\delta \rho}   \}|$ over any geodesic ball $B_{x}(3(1+\frac{1}{b}))$ \emph{lying within the outer region  $\{x \in M : \rho (x) > \frac{R(\delta)}{b}\}$} is bounded above by $C(a,b; \|\phi\|_{\infty})$ $e^{-2 \delta \rho (x)}$.
\end{itemize} 
Next, since $|\nabla\{ |\nabla F|^{2} + A e^{-2\delta \rho}   \}|$ is continuous on $M$, to see the extent to which the above fact can ensure the finiteness of the integral 
$\int_{M} |\nabla\{ |\nabla F|^{2} + A e^{-2\delta \rho}   \}| $, we just have to further decompose the outer region  $\{x \in M : \rho (x) > \frac{R(\delta)}{b}\}$ into a countable number of rings $\{ x \in M : k-1 \leq \rho (x) \leq k + 1\}$, indexed by sufficiently large positive integers $k$; and argue, as in our covering Lemma 4.1 that, due to the lower bound $-b^{2}$ on the sectional curvature $K_{M}$, it is sufficient to use a total of $[\pi e^{bk}] + 1$ geodesic balls with radius $3(1 + \frac{1}{b})$ to cover the ring   
$\{ x \in M : k-1 \leq \rho (x) \leq k + 1\}$, which in turns ensures that (see the proof of proposition 4.4 for more details)

\begin{itemize}
\item the inequality $\int_{B_{x}(3(1+\frac{1}{b}))}|\nabla\{ |\nabla F|^{2} + A e^{-2\delta \rho}   \}| \leq C(a,b; \|\phi\|_{\infty})$ $e^{-2 \delta \rho (x)}$ for any geodesic ball  $B_{x}(3(1+\frac{1}{b})) \subset \{x \in M : \rho (x) > \frac{R(\delta)}{b}\}$ is strong enough to deduce that the integral of $|\nabla\{ |\nabla F|^{2} + A e^{-2\delta \rho}   \}|$ over the outer region $\{x \in M : \rho (x) > \frac{R(\delta)}{b}\}$ is finite (thanks to the condition $0<\frac{b}{2}< a$, which allows $\delta$ to be within the range $\frac{b}{2} < \delta < a$).
\end{itemize}
The above observation gives $\int_{M} |\nabla\{ |\nabla F| + A e^{-2\delta \rho}   \}| < \infty$.  Due to the fact that $\int_{M}e^{-2\delta \rho} < \infty$, which is ensured by the condition $2\delta > b$ (see inequality \eqref{est2}), we finally conclude that $\int_{M} |\nabla\{ |\nabla F|^{2}    \}| < \infty $, which is a \emph{backbone ensuring the correctness of equation \eqref{consequencebackbone}} .
\newline\indent
\subsection{Organization of the article}
In order to make the paper as self-contained as possible in Section \ref{prelim} we collect some facts from the differential geometry and in particular some background specific to the negatively curved manifolds.
\newline\indent
Once that is done, we are ready to establish fundamental statements needed for the proof of Theorems \ref{mainthm}, \ref{thm2}, and their corollaries.  They are:
\begin{itemize}
\item [1)] Exponential decay of the gradient of bounded harmonic functions on negatively curved manifold--Section 3.
\item [2)] Finiteness of $\|\nabla \abs{\nabla F}^{2}\|_{L^{1}(\mathbb{H}^{2}(-a^{2}))}$--Section 4.
\item [3)] Global energy inequality tools--Section 5.
\end{itemize}
Section 6 contains the proofs of the main results. 
\section*{Acknowledgements}
The first author thanks the Department of Mathematics at University of Toronto for its hospitality during the author's visit during which part of this work was carried out as well as
the Institute for Mathematics and its Applications, where the author is in residence.
\section{Preliminaries}\label{prelim}
In this section we gather all the necessary tools from the literature needed in our proof.  A lot of it relies on \cite{redbook}, and we list it here for the convenience of the reader.  We also develop some precise statements regarding the volume forms on the negatively curved manifolds.
\subsection{The Levi-Civita connection, deformation tensor and Ricci curvature}\label{diffgeom}
Here we recall some general background from Riemannian geometry (see for example \cite{Jost}, \cite{Nishikawa}, \cite{Lee}).  In particular, we take a closer look at the deformation tensor mentioned in the introduction.  
\newline\indent 
Let $M$ be an $n$-dimensional complete Riemannian manifold, and $TM$ and $T^{\ast}M$ tangent and cotangent bundles on $M$ respectively.  Let $g$ be a Riemannian metric,  $g (\cdot, \cdot) \in C^{\infty} (M, S^{2}T^{*}M)$, where $S^{2}T^{*}M$ denotes symmetric bilinear forms on $TM$, and $\overline{\nabla}$ the Levi-Civita connection on $(M,g)$,
$$
{\overline{\nabla}} : C^{\infty} (M,TM) \rightarrow C^{\infty}(M, T^{*}M \otimes TM).
$$
Let $\langle \cdot ,\cdot \rangle_{TM\otimes T^{*}M}$ be the natural pairing between vector fields and $1$-forms on $M$.  Given a vector field $X  \in C^{\infty}(M; TM)$, using the metric $g$ we can define $X^{\ast}\in  C^{\infty}(M; T^{\ast}M)$ by
\be\label{lower}
 \langle Y, X^{\ast}\rangle_{TM\otimes T^{\ast}M}=g(X,Y),\quad Y \in  C^{\infty}(M; TM).
\ee
Similarly, given a $1$-form $\omega \in C^{\infty}(M; T^{\ast}M)$ we can define $v_{\omega} \in  C^{\infty}(M; TM)$ by
$$g(v_{\omega}, Y) = \langle Y, \omega\rangle_{TM\otimes T^{\ast}M}.$$ 
Therefore, the Riemannian metric $g$ gives rise to the natural identification $C^{\infty}(M; TM) = C^{\infty}(M; T^{*}M) $.  In particular, if $F$ is a smooth function on $M$, and $d$ is the exterior derivative we have
\be\label{gradFdF}
(\nabla F)^{\ast}=dF.
\ee
Next $g$ also induces its dual metric ${g^{\ast}}(\cdot, \cdot) \in C^{\infty} (M, S^{2}TM)$ by
\[
{g^{\ast}}(\omega,\gamma)=g(v_{\omega},v_{\gamma}),\quad \omega, \gamma \in T^{\ast}M.
\]
Then note
\be\label{samenorm}
\abs{dF}^{2}= {g^{\ast}}(dF,dF)=g(\nabla F,\nabla F)=\abs{\nabla F}^{2}.
\ee
Now by using $g$ again, we can also induce the corresponding positive definite inner product $\overline{g}( \cdot ,\cdot )$ on the bundle  $T^{*}M\otimes T^{*}M $, which is \emph{precisely characterized} by the following condition
\begin{itemize}
\item Let $ e_{1}, e_{2},..., e_{n}$ to be a local orthonormal moving frame for $TM$, and let $\theta^{1} , \theta^{2} , ... \theta^{n}$ to be the corresponding dual frame for $T^{*}M$, then, the list $\{ \theta^{i}\otimes \theta^{j} : 1 \leq j,k \leq n   \}$ is orthonormal with respect to   $\overline{g}( \cdot ,\cdot )$.
\end{itemize}
Next, the Levi-Civita connection $\overline{\nabla}$ on the tangent bundle $TM$ induces the \emph{associated Levi-Civita connection}  $\overline{\nabla}$ on $T^{*}M$ by
\begin{align}\label{dlc}
&\overline{\nabla} : C^{\infty} (M,T^{*}M) \rightarrow C^{\infty}(M, T^{*}M \otimes T^{*}M),\\
&(\overline{\nabla}_{X}\omega)(Y)=(\overline{\nabla}_{X}v_{\omega})^{\ast}(Y),\quad \omega \in T^{\ast}M, X, Y \in TM.
\end{align}
Notice, for the sake of convenience and keeping with the conventions, we use the same abbreviation $\overline{\nabla}$ to denote both the Levi-Civita connection on $TM$ and the induced connection on $T^{*}M$.  No possible confusion should arise since the meaning of the symbol $\overline{\nabla}$ will be clear from the context.  In particular, we have, by the definition of the induced connection $\overline{\nabla}$ on $T^{*}M$, the property that 
\be\label{samenotation}
\overline{\nabla}_{X}Y^{*} = (\overline{\nabla}_{X}Y)^{*},
\ee
for any smooth vector fields $X$, $Y$ on $M$. 
\newline\indent
We now turn our attention to the deformation tensor
\[ 
\Def : C^{\infty}(M, TM) \rightarrow C^{\infty} (M, S^{2}T^{*}M),
\]
defined by (see for example \cite{DindosMitrea})
\begin{equation}\label{defforDefvector}
(Def X)(Y,Z) = \frac{1}{2}\left \{  g(\overline{\nabla}_{Y}X ,Z )  + g(\overline{\nabla}_{Z}X  , Y ) \right\},\quad  Y, Z \in C^{\infty}(M, TM).
\end{equation}
Using the natural identification of the space of vector fields with the space of $1$-forms on $M$ discussed above, the operator $\Def$ can be regarded as the operator from  $C^{\infty}(M, T^{*}M)$ to $C^{\infty} (M, S^{2}T^{*}M)$, and can be defined \emph{alternatively} as follows:
\begin{defn}
For any $1$-form $\omega \in C^{\infty}(M,T^{*}M)$, the deformation tensor $\Def \theta \in C^{\infty} (M, S^{2}T^{*}M)$ is given by 
\begin{equation}\label{defforDefform}
(\Def \omega )(Y,Z) =  \frac{1}{2} \{\langle Z , \overline{\nabla}_{Y}\omega \rangle_{TM\otimes T^{*}M}  +  \langle Y , \overline{\nabla}_{Z}\omega \rangle_{TM\otimes T^{*}M}  \} ,
\end{equation}
for any $Y, Z \in C^{\infty}(M ,TM)$.
\end{defn}
In the sequel we also need the following.  If we express $\omega \in C^{\infty}(M,T^{*}M) $ locally as $\omega = \sum_{a} \omega_{a} dx^{a}$, then $\Def \omega$ can locally be expressed as 
\begin{equation}\label{definc}
Def \omega = \sum_{j,k}\frac{1}{2} (\omega_{j;k} + \omega_{k;j}) dx^{j}\otimes dx^{k},
\end{equation}
where $\omega_{j;k} = \partial_{k}\omega_{j} - \Gamma_{jk}^{l}\omega_{l},$ where $\Gamma_{jk}^{l}$ are the Christoffel symbols.
\newline\indent
Before we go to the next subsection, we briefly discuss the Ricci curvature on a complete $n$-dimensional Riemannian manifold $M$.  Recall, the Ricci curvature is a symmetric tensor $\Ric \in C^{\infty}(M , S^{2}T^{*}M)$ defined as follows
\[
\Ric_{p}(X,X) = \sum_{1\leq i \leq n-1} K_{M}(X,e_{i}),\quad p \in M,
\]
where $e_{1}$, $e_{2}$, ... $e_{n-1}$ are some unit vectors in $T_{p}M$ such that $\{X, e_{1}, ...e_{n-1}\}$ forms an orthonormal basis for $T_{p}M$.
In many occasions, we use the symbol $\Ric(M)$ for $\Ric$. Moreover, if we write $\Ric(M) \geq -b^{2}$, it means that $\Ric(X,X) \geq -b^{2} |X|^{2}$, for all $X \in C^{\infty}(M, TM)$. Moreover, it is clear that, for $2$-dimensional Riemannian manifold $M$, the notion of Ricci curvature $\Ric$ coincides with the sectional curvature $K_{M}$.
\newline\indent
Besides the Ricci curvature tensor $\Ric \in C^{\infty}(M , T^{*}M \otimes T^{*}M)$, we also need to consider the Ricci operator $\Ric : C^{\infty}(M,T^{*}M) \rightarrow C^{\infty}(M,T^{*}M)$ sending the space of $1$-forms into itself, which is defined by
\[
\Ric(u^{\ast}) = \sum_{a,b} \eta^{a} g( R(e_{a} , e_{b})(e_{b}) ,u ),
\]
where  $R(\cdot, \cdot)$ is the Riemannian curvature tensor, $e_{1}$, $e_{2}$, ...$e_{n}$ is a local orthonormal moving frame for $TM$, and $\eta^{1}$, $\eta^{2}$,...$\eta^{n}$ stand for the associated dual frame for $T^{*}M$ with respect to $e_{1}$, $e_{2}$, ...$e_{n}$.

%
In the case of the space form $\mathbb{H}^{n}(-a^{2})$ with sectional curvature $-a^{2}$, we have $R(e_{a} , e_{b})(e_{b}) = -a^{2} e_{a}$, for any local orthonormal frame $e_{1}$, $e_{2}$, ...$e_{n}$ of $T\mathbb{H}^{n}(-a^{2})$. Hence, in particular we have the following fact  
\begin{equation*}
\begin{split}
\Ric(u^{*}) &= \sum_{a,b} \eta^{a} g( R(e_{a} , e_{b})(e_{b}) ,u )\\
& = - a^{2} (n-1)\sum_{a}\eta^{a} g(e_{a},u)\\
& = -a^{2} (n-1) \sum_{a}u^{a}\eta^{a}\\
& = -a^{2} (n-1) u^{*}. 
\end{split}
\end{equation*}
We end this section with a quick summary of basic facts about the Ricci curvature $\Ric \in C^{\infty}(M , T^{*}M \otimes T^{*}M)$ and the Ricci operator $\Ric : C^{\infty}(M,T^{*}M) \rightarrow C^{\infty}(M,T^{*}M)$. In particular,
\begin{align}
\Ric_{p}(X,X)&=K_{M}(p)\abs{X}^{2},\ \ p\in M, X\in  C^{\infty}(M,TM),\  dim M=2,\label{ricci1}\\
\Ric(\omega)&=-(n-1)a^{2}\omega,\ \ \omega \in  C^{\infty}(H^{n}(-a^{2}),T^{\ast}H^{n}(-a^{2})), a>0, n\geq2 .\label{ricci2}
\end{align}

\subsection{Estimates and identities used}\label{moreprelim}

As usual, we start with a complete $n$-dimensional Riemannian manifold $M$, and consider the geodesic normal coordinates on $M$ about a selected base point $O$. 
One of the fundamental properties of the normal coordinates, which we use in computations, is that the Christoffel symbols all \emph{vanish at $O$}:
\be\label{gamma0}
\Gamma^{i}_{jk}=0,
\ee
(see for example \cite{Lee} for more on normal coordinates).
\newline\indent
In the case of a \emph{complete, simply connected, $n$-dimensional Riemannian manifold $M$ with sectional curvature satisfying $-b^{2} \leq K_{M} \leq \-a^{2} < 0$}, the Cartan-Hadamard theorem ensures that the geodesic normal coordinates on $M$ about any selected base point $O \in M$ must be globally defined, which also implies in particular that $M$ is \emph{diffeomorphic} to $\mathbb{R}^{n}$. Moreover, in this case, between any two points $p$, $q$ in such a Riemannian manifold $M$, the geodesic joining $p$ and $q$ is \emph{unique}, and hence we just use the symbol $\overline{pq}$ to denote the unique geodesic joining $p$ and $q$, and $\abs{\overline{pq}}$ stands for the length of the geodesic joining $p$ and $q$.  
\newline\indent
Define the distance function from a point $p \in M$ to a point $x$ by
\[
\rho_{p}(x)\equiv \abs{\overline{px}}.
\]
We usually omit the subscript $p$ and simply write $\rho(x)$ since the base point is clear from the context.
\begin{lemma}[Properties of the distance function{ \cite[Ch. 1]{redbook}}]
The distance function $\rho(x)$ defined as above is smooth\footnote{For more general manifolds $M$, $\rho(x)$ is smooth on $M\setminus Cut(p)$.  See \cite{redbook} for precise definitions and statements.} on all of $M,$ where $M$ is any Riemannian n-manifold such that the exponential map defines the diffeomorphism between $M$ and $\R^{n}$. In addition,  
\begin{align}
\abs{\nabla \rho}^{2}&=1\label{edist1},\\
\Delta \rho&=(n-1) k \coth{k\rho}\quad\mbox{if $M$ has constant sectional curvature $=-k^{2}$,}\label{edist2}\\
\Delta \rho &\geq (n-1)a\cdot \coth (a\rho) \geq (n-1)a\quad\mbox{if $\Ric(M) \leq -a^{2}$},\label{edist3}\\
\Delta \rho & \leq (n-1) b\cdot \coth (b\rho ) \leq (n-1)\frac{1+b\rho}{\rho}\quad\mbox{if $\Ric(M) \geq -(n-1)b^{2}$.} 
\end{align}
\end{lemma}
 
\begin{lemma}\label{angleexpdecayversiontwo}\cite[p.35]{redbook}
Let $M$ be a $n$-dimensional simply connected, complete Riemannian manifold with sectional curvature satisfying $-b^{2} \leq K_{M} \leq -a^{2}$. Let $O \in M$ to be a selected based point, and let $x_{1}$, $x_{2}$ be two points in $M$ for which $|\overline{Ox_{1}}| = |\overline{Ox_{2}}| = R$, for some $R > 1$. Moreover, let $\theta = \angle (\overline{Ox_{1}} , \overline{Ox_{2}})$. Then there exists a sufficiently large universal constant $R_{0} > 1$, \emph{depending only on $a$, $b$, and $n$} such that, whenever $|\overline{Ox_{1}}| = |\overline{Ox_{2}}| = R \geq R_{0}$, we have the following 
\begin{equation}\label{estimateforx1x2} 
2R + \frac{2}{a} (\log \theta -1) \leq |\overline{x_{1}x_{2}}| \leq 2R + \frac{2}{b} (\log \theta + 1).
\end{equation}
    
\end{lemma}
\begin{lemma}\label{integralforsubharmonic}\cite[p.78]{redbook}
Let $M$ be a general $n$-dimensional complete Riemannian manifold. Suppose that $f : B((1+ \tau)R) \rightarrow [0,\infty)$ is a non-negative subharmonic function on $B((1+\tau)R)$ (i.e., $f\geq 0$ and $\Delta f \geq 0$ on $B((1+\tau)R)$). Then 
\begin{equation}
\int_{B(R)} |\nabla f|^{2} \leq \frac{C}{\tau^{2}R^{2}} \int_{B((1+\tau)R)} f^{2} ,
\end{equation}
where $C$ is some universal constant. 
\end{lemma}

\begin{thm}[Gradient Estimate \cite{Yau75, redbook}]\label{thmgradient}
Let M be an n-dimensional complete Riemannian manifold with $\Ric(M)\geq-(n-1)K$, for some constant $K\geq 0$.  If $u$ is a positive harmonic function on $M$ and $B_{r}(x)$ is a geodesic ball in $M$, then
\begin{align}
\abs{\label{gradest} \nabla u} \leq C_{n}\left(\tfrac{1+r\sqrt{K}}{r}\right) u \quad\mbox{on} \ B_{\tfrac r2}(x),
\end{align}
where $C_{n}$ is a constant depending only on $n$.
\end{thm}
\begin{lemma}\cite[p. $15$]{redbook}
Let $M$ be a Riemannian manifold.  Then the following holds in normal coordinates at $x$
\begin{equation}\label{important}
\Delta [ |\nabla F|^{2}](x) = 2 \sum [\partial_{i}\partial_{j}F]^{2}(x) + 2\sum \partial_{i}F(x) \partial_{i} (\Delta F)(x) + 2\Ric (\nabla F , \nabla F)(x).
\end{equation}
\end{lemma}

\subsection{Comparison theorem for Jacobi fields and the growth rate of the volume form on negatively curved Riemannian manifold}\label{Jacobidisscussion}
In this subsection, we only focus on a complete, simply connected, $2$-dimensional Riemannian manifold $M$ with sectional curvature satisfying $-b^{2} \leq K_{M} \leq -a^{2} < 0$  
To begin, let $O$ be a selected point in $M$, and let $exp_{O} : T_{O}M \rightarrow M $ be the \emph{global} exponential map at $O$, whose existence is ensured as before by the Cartan-Hadamard theorem. We remark that the tangent space $T_{O}M  $ is identified with the Euclidean space $\mathbb{R}^{2}$. \\

Let $(\overline{r} , \overline{\theta})$ be the \emph{ordinary polar coordinates on $\mathbb{R}^{2}$ in the Euclidean sense}.  That is, the respective induced vectors 
$\frac{\partial}{\partial \overline{r}}$ and $\frac{\partial}{\partial \overline{\theta }}$ are given by 
\begin{equation}
\begin{split}
\frac{\partial}{\partial \overline{r}}\bigg |_{(x,y)} & = \frac{(x,y)}{(x^{2} + y^{2})^{\frac{1}{2}}},\\
\frac{\partial}{\partial \overline{\theta }}\bigg |_{(x,y)} & =  (-y,x).
\end{split}
\end{equation}
Then, the \emph{geodesic normal polar coordinates} $(r, \theta )$ \emph{ on M} (as induced by  $exp_{O} : T_{O}M \rightarrow M $) is given by the composite function $(r,\theta ) = (\overline{r} , \overline{\theta}) \circ \{exp_{O}\}^{-1} $. \\

Let $v \in T_{O}M = \mathbb{R}^{2}$ be any unit vector, and consider the geodesic $c : [0, \infty ) \rightarrow M$ with $c(0) = O$ and $c'(0) = v$. Then, we notice that 
$\frac{\partial}{\partial r}_{c(r)} = c'(r)$. Next, in order to compute $\frac{\partial}{\partial \theta }_{c(r)} $, we first observe that \emph{the ordinary Euclidean}  $\frac{\partial}{\partial \overline{\theta} }_{rv} $ is given by 

\begin{itemize}
\item  $\frac{\partial}{\partial \overline{\theta} }_{rv} =  r w$, in which $w$ is the unique \emph{unit} vector in $\mathbb{R}^{2}$ such that the pair $\{v, w\}$ forms a positively oriented orthonormal basis for $\mathbb{R}^{2}$. (Recall that we have the identification $ T_{O}M = \mathbb{R}^{2} $, so, we may just regard $v \in T_{O}M$ to be a vector in $\mathbb{R}^{2}$.) 
\end{itemize}

Now, let $Y(r)$ be the Jacobi field along the geodesic $c : [0, \infty ) \rightarrow M$ with $Y(0) = 0$, and $\nabla_{\frac{\partial }{\partial r}}Y_{(0)} =  w$. (Recall that $Y(r)$ is a Jacobi field means that the equation $\overline{\nabla}_{c'}\overline{\nabla}_{c'}Y + R(Y,c')c' = 0$ holds along the geodesic ray $c(t)$.) Then, it is well known that \cite{Jost}
\begin{equation}
Y(r) = (D exp_{O})_{rv} ( r w) =  (D exp_{O})_{rv} (\frac{\partial}{\partial \overline{\theta }}_{rv}      ) ,
\end{equation}
which implies that $ \frac{\partial}{\partial \overline{\theta }}_{rv}  = (D exp_{O}^{-1})_{c(r)} (Y(r))  $. Hence, we have for any $f \in C^{\infty} (M)$
\begin{equation}
\begin{split}
\frac{\partial}{\partial \theta }_{c(r)} f & =  \frac{\partial}{\partial \overline{\theta }}_{rv} (f \circ exp_{O}) \\
& = [(D exp_{O}^{-1})_{c(r)} (Y(r))] (f \circ exp_{O} ) = d (f \circ exp_{O})_{rv} \left( (D exp_{O}^{-1})_{c(r)} (Y(r)) \right) \\
& = (df)_{c(r)} (Y(r)) = Y(r) (f) .
\end{split}
\end{equation}
This shows that we have the important identity $$ \frac{\partial}{\partial \theta }_{c(r)} = Y(r).  $$ 



With the above preparation, we can now discuss the growth rate of the volume form on a complete, simply connected, $2$-dim Riemannian manifold with $-b^{2} \leq K_{M} \leq -a^{2} < 0$. Under the geodesic normal coordinates $(r,\theta )$, the volume form is given by $$\abs{\frac{\partial}{\partial\theta}} dr d\theta,$$  but sometimes we write $\abs{\frac{\partial}{\partial\theta}} d\rho d\theta$ in the case when the distance function $\rho$ from $O$ is used to replace the symbol $r$.
Then, the following comparison theorem in differential geometry is used to give us the growth rate of the volume form.
\begin{thm}\label{JacobifieldThm}
[Comparison theorem for Jacobi fields \cite{Jost}] Let $M$ be a simply connected, complete, $n$-dim Riemannian manifold $M$ with $-b^{2} \leq K_{M} \leq -a^{2} < 0$. Let $c : [0, \infty ) \rightarrow M$ be a geodesic ray in $M$, and we consider a given Jacobi field $Y(r)$ along the geodesic ray $c$ \emph{which is orthogonal to $c$ and satisfies $|Y(0)| =1$}. Then, it follows that we have the following estimate for the growth of $|Y(r)|$ along $c$, for all $r\geq 0$.

\begin{equation}
\frac{1}{a} \sinh(a r) \leq |Y(r)| \leq \frac{1}{b} \sinh(br).
\end{equation}

\end{thm}
Now, by the above comparison theorem together with the identity $ \frac{\partial}{\partial \theta }_{c(r)} = Y(r)  $, we immediately deduce that, for a complete, simply connected Riemannian manifold $M$ of dimension $2$ with $-b^{2} \leq K_{M} \leq -a^{2} < 0$, the weight $G(r,\theta) = \abs{\frac{\partial}{\partial\theta}} $ of the volume form  $\abs{\frac{\partial}{\partial\theta}} dr d\theta$ on $M$ is \emph{at most of the order $\frac{1}{b} \sinh(br)$}. This is used in Section 4 to estimate the integral of a certain non-negative function on $M$. Finally, we also remark that, in the special case of a space form $\mathbb{H}^{2}(-a^{2})$ with constant sectional curvature $-a^{2}$, we have \emph{exactly} $\abs{\frac{\partial}{\partial\theta}} = |Y(r)| = \frac{1}{a} \sinh(a r) $.

\subsection{Geodesic balls, cones and the geometric boundary $S(\infty)$}\label{balls}
In this subsection, we will consider \emph{only} simply connected, completed, $n$-dimensional Riemannian manifold $M$ with sectional curvature satisfying $-b^{2} \leq K_{M} \leq -a^{2} > 0$. Recall that the classical Cartan-Hadamard theorem ensures that the geodesic normal coordinate at any point $O$ in such a negatively curved manifold $M$ is globally defined, and hence $M$ is diffeomorphic to $\mathbb{R}^{n}$. 
\newline\indent
Using the distance function from the Section \ref{moreprelim}, define
a geodesic ball in $M$ with radius $R$ and centered at $x$ by
$$B_{R}(x)=\{ y \in M : \rho_{x}(y)\leq R \}.$$
Next, let $O\in M$ and $v \in T_{O}M$.  Define the cone about $v$ with angle $\theta$ by
 $$C_{O}(v,\theta)=\{y\in M: \angle(v,\overline{Oy})\leq \theta\}.$$
Finally the geometric boundary, the sphere at infinity $S(\infty)$ is
$$S(\infty)=\ \mbox{the set of all geodesic rays from}\  O,$$
which can be canonically identified with the unit sphere in $T_{O}M$:
$S_{O}(1)=\{ v \in T_{O}M : \abs{v}= 1 \}$,
so that every unit vector $v\in T_{O}M$ can then be regarded as a point in $S(\infty)$, and we can simply write $v \in S(\infty)$ (See \cite[p.36]{redbook}, \cite{Anderson, AndersonSchoen}).

\subsection{Bounded harmonic functions on negatively curved manifolds}\label{harmonic}
Anderson \cite{Anderson} and Sullivan \cite{Sullivan} independently, and using different methods, proved the following theorem.
\begin{thm}\cite{Anderson, Sullivan} \label{thmharmonic}
Let $M$ be a simply connected, $n$-dimensional, complete Riemannian manifold with sectional curvature $K_{M}$ satisfying $-b^{2}\leq K_{M}\leq -a^{2}<0$.  Then there exists a unique solution $u\in C^{\infty}(M)\cap C^{0}(\bar M)$ to the Dirichlet problem
\begin{align*}
&\Delta u=0 \ \ \mbox{in}\ \ M,\\
&u\big |_{S(\infty)}=\phi \in  C^{0}(S(\infty)).
\end{align*}
\end{thm}
A simpler proof is also presented in the comprehensive work of Anderson and Schoen  \cite{AndersonSchoen}, and it is also exposed in \cite{redbook}.
The main idea there is to construct two barrier functions and use the Perron's method\footnote{We note that the proof in  \cite{Anderson} also relies on the Perron's method.}.   This in turn is accomplished in three steps, whose conclusions we use in the proof of Proposition \ref{exponentialdecayofgradientF}), which is a crucial tool for our result.  Therefore, we give a brief outline of the proof in \cite[p. 37]{redbook}, and list the needed conclusions:
\begin{itemize}
\item [Step 1)]  Extend the function $\phi$ to all of $M$ and show
\be\label{step1}
\sup_{y\in B_{x}(1)}\abs{\phi(y)-\phi(x)}\leq C_{0}\norm{\phi'}_{L^\infty(S(\infty))}e^{-a\rho(x)}
\ee
To extend $\phi$ to all of $M$ we pick a base point $O \in M$ and use the the geodesic normal polar coordinates $(r,\theta )$ at $O$ to define
\[
\phi (r , \theta ) = \phi (\theta ),\ \mbox{for all}\ r > 0.
\]
Lemma \ref{angleexpdecayversiontwo} is then used to show \eqref{step1}. 
\item [Step 2)] The Laplacian of the average of $\phi$ has an exponential decay.  More precisely, let
\[
\bar \phi (x)=\frac{\int \chi(\rho^{2}_{x}(y))\phi(y)dy}{\int \chi(\rho^{2}_{x}(y))dy},
\]
where $\chi$ is a standard cut off function.  Then it can be showed
\be\label{step2}
\abs{\Delta \bar \phi(x)}\leq C_{0}\norm{\phi'}_{L^\infty(S(\infty))}e^{-a\rho(x)},
\ee
where $\rho(x)$ is the distance function defined in Section \ref{moreprelim}.
\item[Step 3)]  Show there exists $\alpha>0$ and $\delta$ small enough such that  
\begin{align}\label{barrier}
\Delta [\overline{\phi} - \alpha e^{-\delta \rho} ] \geq 0 \quad\mbox{and}\quad \Delta [\overline{\phi} + \alpha e^{-\delta \rho} ] \leq 0.
\end{align}
Then by Perron's\footnote{See \cite{Anderson} for the application in this context or \cite{GilbargTrudinger}.} method there exists a harmonic function $F$ such that 
 \begin{equation}\label{step3}
\overline{\phi} - \alpha e^{-\delta \rho} \leq F \leq \overline{\phi} + \alpha e^{-\delta \rho}.
\end{equation}
The boundary conditions are easily checked.  
\end{itemize}
\subsubsection{Constants $\alpha$ and $\delta$.}\label{constants}
Constants $\alpha$ and $\delta$ from Step 3 play a very important role in our proofs.  Therefore we take some time now to discuss $\alpha$ and $\delta$ and and how they relate to the function $\phi$ and the curvature $-b^{2}\leq K_{M}\leq -a^{2}$.  We emphasize that this exposition is completely based on \cite[p. 40]{redbook} although the details of \eqref{barrier1} and \eqref{barrier2} below were not exposed there. 
\newline\indent
First start with some $\delta>0$ to be specified later.  Using \eqref{edist1} we then observe
\begin{equation}\label{step3a}
\Delta (e^{-\delta \rho(x)}) = \delta e^{-\delta \rho(x)} (\delta - \Delta \rho(x)). 
\end{equation}
Also, by \eqref{edist3} 
 \[
 \Delta \rho \geq (n-1)a\cdot coth (a\rho) \geq (n-1)a
 \] 
Next, one has to choose sufficiently small $\delta$ and sufficiently large $\alpha$ so that \eqref{barrier} does indeed hold.  Let $\delta<a$, then the first equation \eqref{barrier} is obtained as follows.  By \eqref{step2} and \eqref{step3a}
\begin{equation}\label{barrier1}
\begin{split}
 \Delta [\overline{\phi} - \alpha e^{-\delta \rho} ] & = \Delta \overline{\phi}  - \alpha \delta e^{-\delta \rho} (\delta - \Delta \rho) \\
& \geq - C_{0} \|\phi '\|_{\infty} e^{- a\rho } + \alpha \delta [(n-1)a - \delta ] e^{- \delta \rho} \\
& \geq \{ \alpha \delta [(n-1)a - \delta ] - C_{0} \|\phi '\|_{\infty} \} e^{-\delta \rho},
\end{split}
\end{equation}
where $\delta < a$ is used to obtain the last line.  Similarly
\begin{equation}\label{barrier2}
\begin{split}
\Delta [\overline{\phi} + \alpha e^{-\delta \rho} ] & \leq  C_{0} \|\phi '\|_{\infty} e^{- a \rho } -  \alpha \delta [(n-1)a - \delta ] e^{- \delta \rho} \\
& \leq \{C_{0} \|\phi '\|_{\infty} - \alpha \delta [(n-1)a - \delta ]\} e^{- \delta \rho}.
\end{split}
\end{equation}
So, for any $ \delta < a$, we choose $\alpha = \frac{2 C_{0} \|\phi '\|_{\infty} }{ \delta [(n-1)a - \delta ]}$.  Note, in order to guarantee $\alpha>0$ for $n=2$, we need $\delta<a$ and not just $\delta \leq a$. Then \eqref{barrier} follows as needed. 
\newline\indent
In addition, besides $\delta$ not being too large, \emph{we} eventually need $\delta$ not to be too small.  More precisely, when we want to obtain that $\nabla F$ is in $L^{2}(\H)$, we impose additional condition  $\frac a2<\delta$.  Then by the discussion in the Section \ref{Jacobidisscussion} the exponential decay of $\nabla F $ obtained in Proposition \ref{exponentialdecayofgradientF} below will be sufficient to give $\|\nabla F\|_{L^{2}(\H)} < \infty$. 
\newline\indent
Similarly, when we want to obtain that $\nabla F$ is in $L^{2}(M),$ where $M$ is complete, simply connected $2$-dim manifold with sectional curvature satisfying $-b^{2}\leq K_{M}\leq -a^{2}$, where $a, b>0$ and $\frac b2<a$, then there we require $\delta > \frac b2$.
\subsection{Hodge Star operator and Hodge Laplacian}\label{hodge}
Let $d$ denote the exterior differentiation operator, which sends $k$ forms to $k+1$ forms.  As is well-known, $d$ satisfies
\be\label{dd}
dd\omega=0 \quad\mbox{for any}\ \ k-\mbox{form}\ \omega. 
\ee
Its dual operator, $d^\ast$, is given by
\be\label{dstar}
d^\ast=(-1)^k\ast\ast\ast d \ast,
\ee
where $\ast$ is the Hodge $\ast$ operator  and $k$ comes from $d^\ast$ acting on some given $k$-form (see for example \cite{Roe}).  We note $d^{\ast}$ sends $k$ forms to $k-1$ forms.  However,  the only main fact that we need to know in this paper, besides \eqref{dd} and \eqref{simple} below is  the definition of the Hodge Laplacian:
\[
-\Delta \omega=(dd^{\ast}+d^{\ast}d)\omega.
\]
When $\Delta$ acts on a function $F$, then the expression simplifies to
\be\label{DF}
-\Delta F= d^{\ast}dF.
\ee
So for example if we have a function $F$ that is harmonic, and if we define a $1$-form $U$ by
\[
U=dF,
\]
then it is very easy to see that $U$ is a harmonic $1$-form since
\begin{equation}\label{simple}
\begin{split}
(dd^{\ast}+d^{\ast}d)U &=dd^{\ast}dF+d^{\ast}ddF\\
				       &=dd^{\ast}dF\qquad\qquad\qquad \mbox{by}\ \eqref{dd}\\
				       &=d(d^{\ast}dF)\\
				       &=0,
\end{split}
\end{equation}
where in the last line we used the fact that $F$ is harmonic and \eqref{DF}.   The construction of our non-unique solution relies on this simple observation.

\section{Exponential decay of the gradient of a bounded harmonic functions on negatively curved manifold}
The main result of this section is the following proposition.
\begin{prop}\label{exponentialdecayofgradientF}
Let $M$ be an $n$-dimensional complete, simply connected Riemannian manifold with sectional curvature satisfying $-b^{2} \leq K_{M} \leq -a^{2}$. Let $\phi \in C^{1}(S(\infty))$ be any boundary data, and $F \in C^{\infty}(M) \cap C^{0}(\overline{M})$ be the unique bounded harmonic function on $M$ with $F|_{S(\infty)} = \phi$.  Let $\delta<a$. Then, the following inequality holds
\begin{equation}\label{edecay}
|\nabla F|(x)  \leq C_{0} \{1 +  \frac{1}{\delta [(n-1)a - \delta ] } \}  \|\phi '\|_{\infty}e^{-\delta\rho(x)} \ \ \forall x \in M,
\end{equation}
where $C_{0}$ depends only on $a$, $b$, and $n$.
\end{prop}

\begin{remark} In the proof, we use $C_{0}$ to denote a generic constant, which may change from line to line, but it always depends only on $a$, $b$, and the dimension $n$ of the Riemannian manifold. 
\end{remark}

\begin{proof}
Given $\phi \in C^{1}(S(\infty))$ by Theorem \ref{thmharmonic} there exists a unique harmonic function $F \in C^{\infty}(M) \cap C^{0}(\overline{M})$ with $F|_{S(\infty)} = \phi$.  By \eqref{step3} we also have
\begin{equation}\label{Perron}
\overline{\phi} - \alpha e^{-\delta \rho} \leq F \leq \overline{\phi} + \alpha e^{-\delta \rho},
\end{equation}
where $\bar \phi$ is as in \eqref{step2} and $\delta<a$ and $\alpha = \frac{2 C_{0} \|\phi '\|_{\infty} }{ \delta [(n-1)a - \delta ]}$ as discussed in Section \ref{constants}.  Let $x \in M$ and consider two cases. 
\newline
\noindent
{\textbf{Case 1: $\rho(x) > 1$.}}
\newline\indent
Consider a ball $B_{x}(1)$. By \eqref{Perron} 
\begin{equation*}
\begin{split}
\osc_{B_{x}(1)} F & \mathrel{\mathop:}= \sup_{B_{x}(1)} F - \inf_{B_{x}(1)} F \\
& \leq  sup_{B_{x}(1)} (\overline{\phi} + \alpha e^{-\delta \rho}) - inf_{B_{x}(1)} (\overline{\phi} - \alpha e^{-\delta \rho}).
\end{split}
\end{equation*}
Next since $ \inf_{B_{x}(1)} [ \overline{\phi} - \alpha e^{-\delta \rho}  ] \geq  \inf_{B_{x}(1)}  \overline{\phi} - \alpha \cdot \sup_{B_{x}(1)} e^{-\delta \rho}$ and 
$ \sup_{B_{x}(1)} e^{-\delta \rho} = e^{-\delta (\rho(x) - 1)}  $, it follows that
\begin{align*}
\osc_{B_{x}(1)} F \leq  \osc_{B_{x}(1)} \overline{\phi} + 2 \alpha e^{-\delta (\rho(x) - 1)},
\end{align*}
which implies that the following inequality is valid on $B_{x}(1)$ 
\begin{equation*}
0 \leq F - \inf_{B_{x}(1)}F \leq  \osc_{B_{x}(1)} F \leq  \osc_{B_{x}(1)} \overline{\phi} + 2 \alpha e^{-\delta (\rho(x) - 1)}.
\end{equation*}
Now, notice that since $F - inf_{B_{x}(1)}F$ is a \emph{positive} harmonic function on $B_{x}(1)$ we can apply the gradient estimate, Theorem \ref{thmgradient},  to deduce the following inequality \emph{for any $y \in B_{x}(\frac{1}{2})$ }
\begin{equation}
\begin{split}
|\nabla F| (y) = |\nabla [F(y) - inf_{B_{x}(1)}F]|  & \leq C_{0} (1+a) [F(y) - inf_{B_{x}(1)}F] \\
& \leq  C_{0} \{ osc_{B_{x}(1)} \overline{\phi} + 2 \alpha e^{-\delta (\rho(x) - 1)}\}\\
& \leq  C_{0} [C_{0} \|\phi '\|_{\infty} + 2 \alpha e^{\delta}] e^{-\delta \rho (x)} \\
& =  C_{0} \{C_{0} \|\phi '\|_{\infty} +  \frac{4 C_{0} \|\phi '\|_{\infty} }{ \delta [(n-1)a - \delta ]} e^{\delta} \} e^{-\delta \rho (x)}.
\end{split}
\end{equation}
So, in particular, if we choose $y = x$ in the above inequality, we have the important conclusion
\begin{equation}\label{gradientinequalityone}
|\nabla F|(x) \leq   C_{0} \{C_{0} \|\phi '\|_{\infty} +  \frac{4 C_{0} \|\phi '\|_{\infty} }{ \delta [(n-1)a - \delta ]} e^{\delta} \} e^{-\delta \rho (x)}\quad \forall x \in M-B_{o}(1) .
\end{equation}
We now consider the case of $x \in B_{0}(1)$. 
\newline\noindent
{\textbf{Case 2: $\rho(x)\leq 1$.}}
\newline\indent
Here we have $e^{-a} \leq e^{-a \rho (x)}$, and $\sup_{B_{x}(1)} e^{-\delta \rho} = 1$.  Hence 
\begin{align*}
\begin{split}
\osc_{B_{x}(1)} F & \leq  \osc_{B_{x}(1)} \overline{\phi} + 2 \alpha \sup_{B_{x}(1)} e^{-\delta \rho}\\
& \leq  C_{0} \|\phi '\|_{\infty} e^{-a \rho} + 2 \alpha \\
& =  C_{0} \|\phi '\|_{\infty} e^{-a \rho} + 2 \alpha e^{a} e^{-a}\\
& \leq [C_{0}\|\phi '\|_{\infty} + 2 \alpha  e^{a}]e^{-a\rho}\\
& \leq  [C_{0}\|\phi '\|_{\infty} + 2 \alpha e^{a}]e^{-\delta\rho}\\
&= \{C_{0}\|\phi '\|_{\infty} +  \frac{4 C_{0} \|\phi '\|_{\infty} }{ \delta [(n-1)a - \delta ]}e^{a} \}e^{-\delta\rho}
\end{split}
\end{align*}
Next, as in Case 1 we can apply the gradient estimate, Theorem \ref{thmgradient}, to  $F - inf_{B_{x}(1)}F$ to obtain for any $y \in B_{x}(\frac{1}{2})$
\begin{equation}
\begin{split}
|\nabla F| (y) = |\nabla [F(y) - \inf_{B_{x}(1)}F]|  & \leq C_{0} (1 + a)[F(y) - \inf_{B_{x}(1)}F] \\
& \leq  C_{0}\osc_{B_{x}(1)} F \\ 
& \leq  C_{0} \{C_{0}\|\phi '\|_{\infty} +  \frac{4 C_{0} \|\phi '\|_{\infty} }{ \delta [(n-1)a - \delta ]}e^{a}  \} e^{-\delta\rho} .
\end{split}
\end{equation}
By taking $y=x$ in the above inequality, we deduce 
\begin{equation}\label{gradientinequalitytwo}
|\nabla F|(x) \leq  C_{0}  \{C_{0}\|\phi '\|_{\infty} + \frac{4 C_{0} \|\phi '\|_{\infty} }{ \delta [(n-1)a - \delta ]} e^{a} \} e^{-\delta\rho(x)}\quad \forall x \in B_{O}(1) .
\end{equation}
By combining (\ref{gradientinequalityone}) and (\ref{gradientinequalitytwo}) we have that \eqref{edecay} holds for all $x \in M$ as needed. 
\end{proof}
By \eqref{edecay} and the discussion in Section \ref{Jacobidisscussion}, we immediately have the following corollaries
\begin{cor}\label{L2spaceform}
In addition if $\delta >\frac a2$, then 
\[
\norm{\nabla F}_{L^{2}(\H)}<\infty.
\]
\end{cor}
\begin{cor}\label{L2M}
Let $M$ be a complete, simply connected $2$-dim manifold with sectional curvature satisfying $-b^{2}\leq K_{M}\leq -a^{2}$, where $a, b>0$ and $\frac b2<a$, if $\frac b2<\delta <a$, then
\[
\norm{\nabla F}_{L^{2}(M)}<\infty.
\]
\end{cor}

\section{The proof that $\|\nabla \abs{\nabla F}^{2}    \|_{L^{1}(\mathbb{H}^{2}(-a^{2}))}$ is finite }

In the proof of $\|\nabla \abs{\nabla F}^{2}    \|_{L^{1}(\mathbb{H}^{2}(-a^{2}))} < \infty$, we need the assistance of the following geometric lemma, which is itself a consequence of lemma \ref{angleexpdecayversiontwo}.

\begin{lemma}\label{coveringlemma}[Covering Lemma]
Consider $M$ to be a simply connected, complete 2-dimensional Riemannian manifold with sectional curvature $-b^{2} \leq K_{M} \leq -a^{2}$ . Let $O$ be a selected base point in $M$, and let $\rho$ be the distance function from $O$. Then, there exists some sufficiently large universal constant $\overline{R}_{0} > 2$ such that the following assertion holds\\

For any given $R \geq \overline{R}_{0}$, if we take the positive integer $N(R) = [\frac{2\pi}{2e^{-bR}}]+1 = [\pi e^{bR}] + 1$ (here, the symbol $[\lambda]$ means the largest integer $N \in \mathbb{Z}$ with $N \leq \lambda$), then, we can pick a list of vectors $v_{1}, v_{2}, v_{3}, .... v_{N(R)} \in S(\infty) $, which are \emph{evenly distributed on the circle $S(\infty)$} in such a way that we have the following inclusion  
\begin{equation}
\{x \in M : R-1 \leq \rho (x) \leq R+ 1 \} \subset \cup_{i=1}^{N(R)} B_{c_{v_{i}}(R)} \big(3(1+\tfrac{1}{b})\big) ,
\end{equation}
where for each $1 \leq N(R)$, $c_{v_{i}} :[0,\infty) \rightarrow M$ is the geodesic ray of unit speed with $c_{v_{i}}(0) = O$, and $c_{v_{i}}'(0) = v_{i}$.
\end{lemma}

\begin{remark}
In words, what the conclusion of the lemma is saying is that if we consider an annulus in $M$ with inner radius $R-1$ and outer radius $R+ 1$, where $R$ is big enough, then we can cover it by $N(R)$ geodesic balls centered at $c_{v_{i}}(R)$with radius  $\big(3(1+\tfrac{1}{b})\big)$.
\end{remark}

\begin{proof}
To begin, let us select a base point $O$ in $M$, and let $\rho$ be the distance function from $O$.
By Lemma \ref{angleexpdecayversiontwo} there exists a sufficiently large universal constant $R_{0} >1$ such that for any two points $x_{1}$, $x_{2}$ in $M$ with $R = |\overline{Ox_{1}}| = |\overline{Ox_{2}}|$ satisfying $R \geq R_{0}$, we have the following  
\begin{equation}\label{angleexpdecayinequality}
2R + \frac{2}{a} (\log \theta -1) \leq |\overline{x_{1}x_{2}}| \leq 2R + \frac{2}{b} (\log \theta + 1) ,
\end{equation}
where $\theta =  \angle (\overline{Ox_{1}} , \overline{Ox_{2}})$.\\
From now on, we use the universal constant $\overline{R}_{0} = R_{0} + 1$. Now, choose any $R \geq \overline{R}_{0}$, and let $v \in S(\infty)$. We then consider the geodesic ray 
\begin{align*}
&c_{v} : [0,\infty ) \rightarrow M,\\
&c_{v}(0) =O,\quad\mbox{and}\quad c_{v}'(0) = v.
\end{align*} 
Now, we consider the universal angle $\theta^{(R)} = e^{-bR}$, and the sector $T_{O}(v,\theta^{(R)} ; R-1, R+1)$ defined by 
\begin{equation*}
T_{O}(v,\theta^{(R)} ; R-1, R+1) = \{ x \in C_{O}(v, \theta^{(R)}) : R-1 \leq \rho(x) \leq R+1             \} ,
\end{equation*}
where the cone $C_{O}(v, \theta^{(R))}$ was defined in Section \ref{balls}.
Our goal is to prove that $ T_{O}(v,\theta^{(R)} ; R-1, R+1) \subset B_{c_{v}(R)}(3(1+\frac{1}{b}))$. To this end, let $x \in T_{O}(v,\theta^{(R)} ; R-1, R+1) $. Then, $\rho(x) = R + \lambda$, for some $\lambda \in [-1,1]$. By the triangle inequality, we have
\begin{equation}
|\overline{c_{v}(R)x}| \leq |\overline{c_{v}(R) c_{v}(R+\lambda)}|  + |\overline{c_{v}(R+\lambda)x}| \leq |\lambda| + |\overline{c_{v}(R+\lambda)x}| .
\end{equation}
But from \eqref{angleexpdecayinequality} with $x_{1} = c_{v}(R+\lambda)$ and $x_{2} = x$, it follows
\begin{equation}
\begin{split}
|\overline{c_{v}(R+\lambda)x}| & \leq 2(R + |\lambda|) + \frac{2}{b} \{\log [\angle (\overline{Oc_{v}(R+\lambda)} ,\overline{Ox})] + 1\} \\
& \leq 2R + 2 + \frac{2}{b} [ \log (\theta^{(R)}) + 1]\\
& = 2R + 2 + \frac{2}{b} \{ \log (e^{-bR}) + 1 \} \\
& = 2R + 2 + \frac{2}{b} (-bR + 1) \\
& = 2 + \frac{2}{b} .
\end{split}
\end{equation}
Hence
\begin{equation}
 |\overline{c_{v}(R)x}| \leq |\lambda| + |\overline{c_{v}(R+\lambda)x}|  \leq |\lambda| + 2(1+\frac{1}{b}) < 3(1+ \frac{1}{b}) .
\end{equation}
This shows that every $x \in T_{O}(v,\theta^{(R)} ; R-1, R+1)$ must lie in the geodesic ball $B_{c_{v}(R)}(3(1+\frac{1}{b}))$.  
To conclude the proof, we just take the integer $N(R) = [\frac{2\pi}{2e^{-bR}}] + 1$. Then, we can select some \emph{evenly distributed} vectors $v_{1}, v_{2}, ..., v_{N(R)} \in S(\infty)$ such that 

\begin{equation}\label{Sectorcovering}
\{x \in M : R-1 \leq \rho (x) \leq R + 1\} = \cup_{i=1}^{N(R)}  T_{O}(v_{i}, \theta^{(R)} ; R-1, R+1).
\end{equation}
Since we already know that, for each $1 \leq i \leq N(R)$, we have $ T_{O}(v_{i},\theta^{(R)} ; R-1, R+1) \subset B_{c_{v_{i}}(R)}(3(1+\frac{1}{b}))$,
in which $c_{v_{i}} : [0,\infty ) \rightarrow M$ is the geodesic with $c_{v_{i}}(0) = O$ and $c_{v_{i}}'(0) = v_{i}$, it follows at once from relation \eqref{Sectorcovering} that

\begin{equation}
\begin{split}
\{x \in M : R-1 \leq \rho (x) \leq R + 1\} & = \cup_{i=1}^{N(R)}  T_{O}(v_{i}, \theta^{(R)} ; R-1, R+1) \\
& \subset \cup_{i=1}^{N(R)} B_{c_{v_{i}}(R)}(3(1+\frac{1}{b})) ,
\end{split}
\end{equation}

as desired.

\end{proof}

With the help of the Covering Lemma \ref{coveringlemma} and Lemma \ref{integralforsubharmonic}, we can now prove the following fact.

\begin{prop}\label{L1finite}
Let $a, b>0$ satisfy $\frac{1}{2}b < a \leq b$, and let $M$ be a simply connected, complete $2$-dimensional Riemannian manifold with sectional curvature $-b^{2} \leq K_{M} \leq -a^{2}$. Then, for any bounded harmonic function $F \in C^{\infty}(M) \cap C^{0}(\overline{M}),$ which arises from $C^{1}$- boundary data $\phi$, it follows that 
$\int_{M} |\nabla |\nabla F|^{2}| < \infty$.
\end{prop}
\begin{remark} We note that in the proof of Proposition \ref{L1finite}, \emph{it is not necessary for us to obtain a uniform bound of $\int_{M} |\nabla |\nabla F|^{2}| < \infty$ in terms of, say, $\|\phi'\|_{\infty}$}. All we need is just to confirm that the integral $\int_{M} |\nabla |\nabla F|^{2}|  $ \emph{is finite}, because this is already enough to ensure that $\int_{M} div\{\nabla |\nabla F|^{2} \} = 0$.
\end{remark}
\begin{proof}
As usual, we begin with a bounded harmonic function $F \in C^{\infty}(M) \cap C^{0}(\overline{M})$ such that $F\big|_{S(\infty)}=\phi$.  Let $\overline{R}_{0}$ be the sufficiently large universal constant as determined in Lemma \ref{coveringlemma}.
\newline\indent
Since $F$ is smooth on $M$, in order to prove that $\int_{M} |\nabla |\nabla F|^{2}| < \infty$, it is sufficient to see that 
$\int_{M -B_{O}(R)} |\nabla |\nabla F|^{2}| < \infty$, for some large $R > \overline{R}_{0}$, where $O$ is a selected base point in $M$.\\
We first write
\begin{equation}\label{est3}
\begin{split}
|\nabla |\nabla F|^{2}|  & \leq |\nabla \{|\nabla F|^{2} + A e^{-2\delta \rho} \}|  + A |\nabla e^{-2\delta \rho}| \\
& = |\nabla \{|\nabla F|^{2} + A e^{-2\delta \rho} \}| + A (2 \delta ) e^{-2\delta \rho}.
\end{split}
\end{equation}
We estimate the first term on the right.  First, by \eqref{important}
\[
\Delta [ |\nabla F|^{2}](x) = 2  [\partial_{i}\partial_{j}F]^{2}(x) + 2 \partial_{i}F(x) \partial_{i} (\Delta F)(x) + 2Ric (\nabla F , \nabla F)(x).
\]
Since $\Delta F = 0$ it follows from the above formula that

\begin{equation}\label{subharmonicprep}
\Delta [ |\nabla F|^{2}] \geq 2\Ric (\nabla F , \nabla F) \geq -2b^{2} |\nabla F|^{2}.
\end{equation}

To proceed further, we take $\delta $ to be any \emph{fixed choice} of positive number within the range $\frac{b}{2} < \delta < a$ (i.e., we choose such a $\delta$ once and for all) and by Proposition \ref{exponentialdecayofgradientF} we have

\be\label{decay2}
|\nabla F| \leq C_{a,b} \|\phi'\|_{\infty} e^{-\delta \rho},
\ee where the constant $C_{a,b}$ depends only on $n=2$ and $a$, and $b$. Notice that our fixed choice of $\delta \in (\frac{b}{2} , a )$ automatically satisfies the condition $\delta > \frac{a}{2}$, due to the fact that $b \geq a$. Next, since $\delta > \frac{b}{2}$, we can choose some sufficiently large positive number  depending on $\delta$, $R(\delta) > 2$, such that 

\begin{equation}\label{survivalofsecondineq}
\frac{b}{2} < \frac{b}{2} (1 + \frac{1}{R(\delta)} ) < \delta < a \leq b.
\end{equation}

Next, we have to find some $A >0$ large enough, and some sufficiently large radius $R$ such that the function $ |\nabla F|^{2} + A e^{-2\delta \rho}$ will be subharmonic on $\mathbb{H}^{2}-B_{O}(R)$. To achieve this we use \eqref{edist3} (with the condition $K_{M} \geq -b^{2}$) and observe
\begin{equation}
\Delta \rho \chi_{\{\rho \geq \frac{R(\delta)}{b}\}} \leq b\cdot \coth(b\rho) \chi_{\{\rho \geq \frac{R(\delta)}{b}\}} \leq (b + \frac{1}{\rho}) \chi_{\{\rho \geq \frac{R(\delta)}{b}\}} 
\leq b (1 + \frac{1}{R(\delta)}) \chi_{\{\rho \geq \frac{R(\delta)}{b}\}}.
\end{equation}

Hence from \eqref{subharmonicprep}, \eqref{decay2} and \eqref{step3a}
\begin{equation}\label{subharmonicfinal}
\begin{split}
\Delta \{|\nabla F|^{2} + A e^{-2\delta \rho}\} \chi_{\{\rho \geq \frac{R(\delta)}{b}\}} & \geq \{-2b^{2} C_{a,b}^{2} \|\phi'\|_{\infty}^{2} e^{-2\delta \rho} +
 A (2\delta) e^{-2\delta \rho} [2\delta - \Delta \rho] \}\chi_{\{\rho \geq \frac{R(\delta)}{b}\}} \\
& \geq 2 \{ A \delta [2\delta - b (1+\frac{1}{R(\delta)})] - b^{2} C_{a,b}^{2}   \|\phi'\|_{\infty}^{2}    \} e^{-2\delta \rho} \chi_{\{\rho \geq \frac{R(\delta)}{b}\}} \\
\end{split}
\end{equation}
Notice that we definitively have $2\delta - b (1+\frac{1}{R(\delta)}) > 0$, thanks to our choice of $R(\delta)$ which ensures the survival of the second inequality sign in \ref{survivalofsecondineq}. Next, we just take 

\begin{equation}
A = \frac{2b^{2} C_{a,b}^{2}   \|\phi'\|_{\infty}^{2} }{\delta [2\delta - b(1+\frac{1}{R(\delta)})] }. 
\end{equation}

With this choice of $A$, it follows from (\ref{subharmonicfinal}) that 

\begin{equation}
\Delta \{|\nabla F|^{2} + A e^{-2\delta \rho}\} \chi_{\{\rho \geq \frac{R(\delta)}{b}\}} \geq 0. 
\end{equation}

That is, the function $|\nabla F|^{2} + A e^{-2\delta \rho}$ is subharmonic on $M-B_{O}(\frac{R(\delta)}{b})$. So, we may apply Lemma \ref{integralforsubharmonic} to $|\nabla F|^{2} + A e^{-2\delta \rho}$ and deduce that for any geodesic ball $B_{x}(6(1+ \frac{1}{b})) \subset  M - B_{O}(\frac{R(\delta)}{b})$, we have

\begin{equation}\label{crucial}
\begin{split}
&\int_{B_{x}(3(1+\tfrac{1}{b}))} |\nabla \{|\nabla F|^{2} + A e^{-2\delta \rho} \}| \\
&\qquad \leq \abs{B_{x}\big(3(1+\tfrac{1}{b})\big)}^{\frac 12} \left\{\int_{B_{x}(3(1+\tfrac{1}{b}))} |\nabla \{|\nabla F|^{2} + A e^{-2\delta \rho} \}|^{2}\right\}^{\frac{1}{2}} \\
&\qquad \leq C_{b} \left\{\int_{B_{x}(6(1+\frac{1}{b}))} [|\nabla F|^{2} + A e^{-2\delta \rho} ]^{2} \right\}^{\frac{1}{2}} \\
&\qquad \leq C(a,b; \|\phi\|_{\infty}) e^{-2 \delta \rho (x)} ,
\end{split}
\end{equation}

where we again used \eqref{decay2} to go to the last line. We further remark that, in the above estimation, we have implicitly employed the volume comparison theorem in differential geometry which says that $K_{M} \geq -b^{2} $ implies that the volume of any geodesic ball $B_{x}(6(1+\frac{1}{b}))$ in $M$ is bounded above by a \emph{universal constant $C_{b}$ }(such a universal constant $C_{b}$ which serves as the upper bound is indeed the constant volume of any geodesic ball with radius $6(1+\frac{1}{b})$ \emph{in the space form $\mathbb{H}^{2}(-b^{2})$}).
\newline\indent
Now, let us take $K_{0}$ to be a sufficiently large positive integer for which $K_{0} \geq max\{ \overline{R}_{0}, \frac{R(\delta)}{b} + 6 (1+\tfrac{1}{b})\}$, where 
$\overline{R}_{0}$ is the sufficiently large universal constant determined in Lemma \ref{coveringlemma}. Then, by Lemma \ref{coveringlemma}, for any positive integer $k \geq K_{0} $, if we take the positive integer $N(k) = [\pi e^{bk}]+1$ then, we can pick a list of vectors $v_{k,1}, v_{k,2}, v_{k,3}, .... v_{k,N(k)} \in S({\infty}) $ in such a way that we have the following inclusion  
\begin{equation*}
\{x \in M : k-1 \leq \rho (x) \leq k + 1 \} \subset \cup_{i=1}^{N(k)} B_{c_{v_{k,i}}(k)} (3(1+\frac{1}{b})) .
\end{equation*}
By combining inequality (\ref{crucial}) and the above inclusion, it follows that
\begin{equation}\label{est1}
\begin{split}
\int_{\{\rho(x) \geq K_{0}\}} |\nabla \{|\nabla F|^{2} + A e^{-2\delta \rho} \}| & \leq \sum_{k=K_{0}}^{\infty}\int_{\{  k-1 \leq \rho (x) \leq k + 1 \}}|\nabla \{|\nabla F|^{2} + A e^{-2\delta \rho} \}|\\
& \leq \sum_{k=K_{0}}^{\infty} \sum_{i=1}^{N(k)} \int_{B_{c_{v_{k,i}}(k)} (3(1+\frac{1}{b}))}|\nabla \{|\nabla F|^{2} + A e^{-2\delta \rho} \}| \\
& \leq \sum_{k=K_{0}}^{\infty} \sum_{i=1}^{N(k)} C(a,b;\|\phi\|_{\infty}) e^{-2\delta k} \\
& \leq \sum_{k=K_{0}}^{\infty} C(a,b,\|\phi\|_{\infty}) \{[\pi e^{bk}]+1\}  e^{-2\delta k} \\
& < \infty.
\end{split}
\end{equation}
We note that to obtain the last inequality $\sum_{k=K_{0}}^{\infty} C(a,b,\|\phi\|_{\infty}) \{[\pi e^{bk}]+1\}  e^{-2\delta k}  < \infty   $, we use the fact that our fixed choice of $\delta$ lies within the range $\frac{1}{2} b < \delta < a$ , so that $(2\delta -b) > 0$ is automatic, and hence $\sum_{k=K_{0}}^{\infty} e^{bk} e^{-2\delta k} = \sum_{k=K_{0}}^{\infty}e^{-(2\delta -b)k}< \infty$.
\newline\indent
Next, we notice that the volume form on $M$ (with respect to the geodesic normal polar coordinates $(r,\theta)$ about $O$) is in the form of $G(r,\theta) dr d\theta$, where $G(r,\theta) \leq \frac{1}{b}\sinh (b r)$ thanks to the comparison theorem \ref{JacobifieldThm} for Jacobi fields (with $K_{M} \geq -b^{2}$) . Hence, it follows again from $2\delta > b$ that

\begin{equation}\label{est2}
\begin{split}
\int_{M}  e^{-2\delta \rho} & = \int_{0}^{2\pi} \int_{0}^{\infty}e^{-2\delta r}G(r,\theta) dr d\theta \\
& \leq \int_{0}^{2\pi} \int_{0}^{\infty}e^{-2\delta r}  \frac{1}{b}\sinh (b r)  dr d\theta< \infty.
\end{split}
\end{equation}  
Hence by \eqref{est3}, \eqref{est1} and \eqref{est2}
\begin{equation}\label{est4}
\int_{\{\rho(x) \geq K_{0}\}} |\nabla |\nabla F|^{2} | \leq \int_{\{\rho(x) \geq K_{0}\}} |\nabla \{|\nabla F|^{2} + A e^{-2\delta \rho} \}| + A (2 \delta ) \int_{\{\rho(x) \geq K_{0}\}}e^{-2\delta \rho} < \infty.
\end{equation}
Since $|\nabla |\nabla F|^{2} |$ is continuous in $M$, by\eqref{est4} we must have $\int_{M}|\nabla |\nabla F|^{2} | < \infty   $ as needed.
\end{proof}

\section{Finite Dissipation}\label{energysection}
We begin with two propositions, which help us establish the energy inequality \eqref{energyM}.  First, using $\overline{g}( \cdot ,\cdot )$ on $ T^{*}M\otimes T^{*}M $ defined in Section \ref{prelim}, we can consider for each $1$-form $\theta$,  two non-negative valued functions $\overline{g}(  Def \theta  , Def \theta  ) \in C^{\infty}(M)$, and $\overline{g}(  \overline{\nabla} \theta  , \overline{\nabla} \theta  ) \in C^{\infty}(M)$. 
We have the following relationship between them.
\begin{lemma}\label{Deflessnabla}
For any given $n$-dimensional complete Riemannian manifold $M$, we have  
\begin{equation}
\overline{g}(  Def \theta  , Def \theta  ) \leq \overline{g}(  \overline{\nabla} \theta  , \overline{\nabla} \theta  ).
\end{equation}

\end{lemma}

\begin{proof}
Let $p\in M$, and consider the \emph{geodesic normal coordinates} $(x^{1}, x^{2}, ..., x^{n})$ \emph{about the point $p$}, so that the natural frame $\partial_{1}|_{p} , \partial_{2}|_{p} , ... ,\partial_{n}|_{p}$ (induced by the geodesic normal coordinates) \emph{at the point $p$} is orthonormal, and that the Christoffel symbols $\Gamma_{jk}^{l}$ (induced by the geodesic normal coordinate) vanish \emph{at the point $p$}. Hence, for any $1$-form $\theta = \sum_{j} \theta_{j} dx^{j}$ we have $\theta_{j;k}(p) = \partial_{k}\theta_{j}(p)$. So, by \eqref{definc} it follows
\begin{equation*}
\begin{split}
\overline{g}(  Def \theta  , Def \theta  )|_{p} & = \frac{1}{4} \overline{g} ( \sum_{i,j} (\partial_{i}\theta_{j} + \partial_{j}\theta_{i} )(p) dx^{i}\otimes dx^{j}|_{p} ,
 \sum_{k,l} (\partial_{k}\theta_{l} + \partial_{l}\theta_{k} )(p) dx^{k}\otimes dx^{l}|_{p} ) \\
& = \frac{1}{4}   \sum_{i,j} \sum_{k,l}   (\partial_{i}\theta_{j} + \partial_{j}\theta_{i} )(p)(\partial_{k}\theta_{l} + \partial_{l}\theta_{k} )(p) \delta^{ik}\delta^{jl} \\
& = \frac{1}{4} \sum_{i,j} (\partial_{i}\theta_{j} + \partial_{j}\theta_{i} )(p)(\partial_{i}\theta_{j} + \partial_{j}\theta_{i} )(p)\\
& \leq \frac{1}{2} \sum_{i,j} (\partial_{i}\theta_{j}(p))^{2} + (\partial_{j}\theta_{i}(p))^{2} \\
& = \sum_{i,j} (\partial_{i}\theta_{j}(p))^{2}.
\end{split}
\end{equation*}

On the other hand, by \eqref{gamma0} the Christoffel symbols $\Gamma_{jk}^{l}$ vanish at $p$, so it follows that $\overline{\theta}|_{p} = \sum_{i,j} \partial_{i}\theta_{j}(p) dx^{i}\otimes dx^{j}|_{p}$.  Hence 

\begin{equation*}
\overline{g}(  \overline{\nabla} \theta  , \overline{\nabla} \theta  )|_{p} =  \sum_{i,j} (\partial_{i}\theta_{j}(p))^{2},
\end{equation*}
and
\begin{equation*}
\overline{g}(  Def \theta  , Def \theta  )|_{p} \leq \sum_{i,j} (\partial_{i}\theta_{j}(p))^{2} = \overline{g}(  \overline{\nabla} \theta  , \overline{\nabla} \theta  )|_{p}.
\end{equation*}
Since $p \in M$ is arbitrary in the above argument, it follows that the above inequality is valid for all points in $M$ as needed.
 
\end{proof}

\begin{prop}\label{finitedissipation}
Let $a, b>0$ and and such that $\frac{1}{2}b < a \leq b$.  Let $M$ be a simply connected, complete $2$-dimensional Riemannian manifold with sectional curvature $-b^{2} \leq K_{M} \leq -a^{2}$. Let $\phi \in C^{1}(S(\infty))$ be any given boundary data, and let $F \in C^{\infty}(M)\cap C^{0}(\overline{M})$ be the unique bounded harmonic function on $\overline{M}   $ with 
$F|_{S(\infty)}= \phi$. Then the following holds
\begin{equation}\label{finitedissipationeq}
\int_{M}  \overline{g} (\overline{\nabla}(dF) , \overline{\nabla}(dF) ) = -\int_{M} Ric (\nabla F, \nabla F)  \leq b^{2} \int_{M} |\nabla F|^{2}.
\end{equation}

\end{prop}

\begin{proof}
Let $F \in C^{\infty}(M)\cap C^{0}(\overline{M})$ be the unique bounded harmonic function on $M$ with $F|_{S(\infty)}= \phi$, where $\phi \in C^{1}(S(\infty))$ is some given boundary data. 

Then again by \eqref{important}
\begin{equation}\label{importantb}
\Delta ( |\nabla F|^{2})(x) = 2 \sum_{1 \leq i,j \leq 2} [\partial_{i}\partial_{j}F]^{2}(x) + 2\sum_{1 \leq i,j \leq 2}\partial_{i}F(x) \partial_{i} (\Delta F)(x) + 2 \Ric (\nabla F , \nabla F)(x) , 
\end{equation}
where $\partial_{1}, \partial_{2} $ is the natural coordinate frame induced by the geodesic normal coordinates \emph{about the point  $x$.} 
So again by \eqref{gamma0} Christoffel symbols $\Gamma_{jk}^{l}$ vanish at $x$, and we have 
$$\sum_{1 \leq i,j \leq 2} [\partial_{i}\partial_{j}F]^{2}(x) =  \overline{g} (\overline{\nabla}(dF) , \overline{\nabla}(dF) )|_{x}.$$  
Also, as before we use $\Delta F=0$ in \eqref{importantb} to obtain 
\begin{equation*}
\frac{1}{2} \Delta (|\nabla F|^{2})(x) =  \overline{g} (\overline{\nabla}(dF) , \overline{\nabla}(dF) )(x) + \Ric (\nabla F, \nabla F) (x) \quad \forall x \in  M. 
\end{equation*}

Now, for each positive integer $k \geq 1$, consider a smooth function $\psi_{k} : [0,\infty ) \rightarrow \mathbb{R}$, which satisfies $\chi_{[0,2^{k}]} \leq \psi_{k} \leq \chi_{[0,2^{k+1}]}$, and $|\psi_{k}'| \leq \frac{2}{2^{k}}$. Now, let $O$ be a selected base point in $M$, and let $\rho$ be the distance function from $O$. Then, by multiplying the above equality by the cut off function $\psi_{k}(\rho^{2})$ (which is compactly supported in $B_{O}(3k)$) and integrating over $M$, we yield the following equality
\begin{equation}\label{importantidea}
\int_{M} \frac{1}{2} div (\nabla |\nabla F|^{2}) \psi_{k}(\rho^{2}) = \int_{M}  \psi_{k}(\rho^{2})\overline{g} (\overline{\nabla}(dF) , \overline{\nabla}(dF) ) 
+ \int_{M} \psi_{k}(\rho^{2}) \Ric( \nabla F, \nabla F) .
\end{equation}
But since  $\psi_{k}(\rho^{2})$ is compactly supported in $B_{O}(3R)$, it is plain to see that, for every $k \geq 1$, we have
\begin{equation}
\begin{split}
 | \int_{M} \frac{1}{2} div (\nabla |\nabla F|^{2}) \psi_{k}(\rho^{2}) | 
&= \frac{1}{2} | \int_{M} 2\psi_{k}'(\rho^{2}) \nabla \rho \cdot \nabla (|\nabla F|^{2})  | \\
& \leq \frac{2}{2^{k}} \int_{M} |\nabla (|\nabla F|^{2})|. 
\end{split}
\end{equation}
Since, according to Proposition \ref{L1finite}, we have $ \int_{M} |\nabla (|\nabla F|^{2})| < \infty $, it follows from the above inequality that 
\begin{equation}
\lim_{k\rightarrow \infty}  \int_{M} \frac{1}{2} div (\nabla |\nabla F|^{2}) \psi_{k}(\rho^{2}) = 0.
\end{equation}

On the other hand, by the monotone convergence theorem, we have
\begin{equation}
\lim_{k\rightarrow \infty}  \int_{M} \psi_{k}(\rho^{2})\overline{g} (\overline{\nabla}(dF) , \overline{\nabla}(dF) ) = \int_{M}\overline{g} (\overline{\nabla}(dF) , \overline{\nabla}(dF) ),
\end{equation}
and that 

\begin{equation}
\lim_{k\rightarrow \infty}  \int_{M} \psi_{k}(\rho^{2})\cdot (-\Ric (\nabla F, \nabla F)) = \int_{M} (-\Ric (\nabla F, \nabla F))
\end{equation}

As a result, by taking the limit of each side in equality (\ref{importantidea}), we get

\begin{equation}
0 =  \int_{M}\overline{g} (\overline{\nabla}(dF) , \overline{\nabla}(dF) ) + \int_{M} \Ric (\nabla F, \nabla F) .
\end{equation}
 That is, we have 

\begin{equation*}
\int_{M}\overline{g} (\overline{\nabla}(dF) , \overline{\nabla}(dF) ) =  -\int_{M} Ric (\nabla F, \nabla F) \leq b^{2}\int_{M} |\nabla F|^{2} , 
\end{equation*}
in which the last inequality follows from the fact \eqref{ricci1} that $\Ric(\nabla F , \nabla F) = K_{M} |\nabla F|^{2} \geq -b^{2} |\nabla F|^{2}$.
\end{proof}

\begin{cor}\label{finitedissipationonH}
Let $a>0$ and let $\phi \in C^{1}(S(\infty))$ be any given boundary data, and let $F \in C^{\infty}(\H)\cap C^{0}(\overline{\H})$ be the unique bounded harmonic function on $\overline{\H}   $ with 
$F|_{S(\infty)}= \phi$. Then the following holds
\begin{equation}
\int_{\H}  \overline{g} (\overline{\nabla}(dF) , \overline{\nabla}(dF) ) = a^{2}\int_{\H} \abs{\nabla F}^2 . 
\end{equation}

\end{cor}
\begin{proof}
This is immediate from equation \eqref{finitedissipationeq} since $ \Ric (\nabla F, \nabla F)=-a^{2}\abs{\nabla F}^2.$ 
\end{proof}
\section{Proofs of  the main results}\label{proofs}
First we establish the following lemma.  A simpler computation in normal coordinates could also be done in the same spirit as the computation in the Euclidean space.  However, we present a different proof below due to its intrinsic nature.
\begin{lemma}\label{convection} The following identity is valid for any smooth function $f$ on any given $n$-dim Riemannian manifold $M$.
\begin{equation}
\overline{\nabla}_{\nabla f}df = \frac 12 d\abs{df}^{2}.
\end{equation}
\end{lemma}

\begin{proof}
First, for any smooth vector field $X$ on a Riemannian manifold $M$, and for any smooth function $f$ on $M$, we write  $$X(f) = \langle X , df \rangle_{TM\otimes T^{*}M}.$$ Next, recall that the Lie bracket $[X,Y]$ between two vector fields $X$ and $Y$, is itself another vector field, and is characterized by $[X,Y](f) = X(Y(f)) - Y(X(f))$, for any $f \in C^{\infty}(M)$.
\newline\indent
Now, for a given smooth function $f$ on a Riemannian manifold $M$, we consider the gradient field $W = \nabla f$, which means that $W^{*} = df$. Then, by \eqref{lower} it is plain to see that 
\be\label{eq6b}
W(f) = \langle W , df \rangle_{TM\otimes T^{*}M} = g(W , W) = |W|^{2}.
\ee  
Next, we have the following identity for any smooth vector field $X$ on $M$, due to the fact that the Levi-Civita connection $\overline{\nabla}$ on $TM$
is compatible with the Riemannian metric $g(\cdot , \cdot )$ on $M$ we have
\begin{equation}\label{eq6a}
\frac{1}{2} X(|W|^{2}) =\frac{1}{2} \langle X , d(|W|^{2}) \rangle_{TM\otimes T^{*}M} = g(\overline{\nabla}_{X}W , W),
\end{equation}
and
\begin{align}
\langle X , \overline{\nabla}_{W}W^{*} \rangle_{TM\otimes T^{*}M} & = \langle X , (\overline{\nabla}_{W}W)^{*} \rangle_{TM\otimes T^{*}M}\quad\mbox{by \eqref{samenotation}}\nonumber\\
& = g(X, \overline{\nabla}_{W}W ) \quad\mbox{by \eqref{lower}}\nonumber\\
& = W( g(X,W) ) - g(\overline{\nabla}_{W}X,W ) \quad\mbox{by compatibility}\nonumber\\
& = W( \langle X , df\rangle_{TM\otimes T^{*}M}       ) - g(\overline{\nabla}_{W}X,W )\nonumber \\
& = W(X(f)) - g(\overline{\nabla}_{W}X,W ).\label{lastline}
\end{align}
But due to the torsion free property of the Levi-Civita connection $\overline{\nabla}$ on $TM$, which says that $-\overline{\nabla}_{W}X = [X,W] - \overline{\nabla}_{X}W$, we have
\[
-g(\overline{\nabla}_{W}X,W )=g([X,W], W)-g(\overline{\nabla}_{X}W,W).
\]
Hence by \eqref{lastline} 
\begin{equation*}
\begin{split}
\langle X , \overline{\nabla}_{W}W^{*} \rangle_{TM\otimes T^{*}M} &= W(X(f)) + g([X,W] , W) - g(\overline{\nabla}_{X}W,W) \\
& = W(X(f)) + [X,W](f) -  \frac{1}{2} X(|W|^{2})\quad\mbox{by \eqref{lower} and \eqref{eq6a}}\\
& = W(X(f)) + X(W(f)) - W(X(f)) -  \frac{1}{2} X(|W|^{2}) \\
& = \frac{1}{2} X(|W|^{2}),
\end{split}
\end{equation*}
where the last equality follows since $X(W(f)) = X(|W|^{2})$ by \eqref{eq6b}.   In conclusion, by using \eqref{eq6a} again, the following equality holds for any smooth vector field $X$ on $M$
\begin{equation*}
 <X , \overline{\nabla}_{W}W^{*} - \frac{1}{2}d(|W|^{2}) >_{TM\otimes T^{*}M} = 0 ,
\end{equation*}
which means the same as saying that $\overline{\nabla}_{\nabla f}df - \frac{1}{2}d(|df|^{2}) = 0$ as needed.
\end{proof}

\subsection{Proof of Theorem \ref{mainthm}.}  First we show existence and the lack of uniqueness.\\
{\textbf {Existence and Non-uniqueness:} }
For convenience we recall the Navier-Stokes equation on $\H$.
\begin{equation}\tag{$\NS{\H}$}
\begin{split}
\partial_{t}U^{*} - \Delta U^{*}  + \overline{\nabla}_{U}U^{*} - 2\Ric(U^{\ast})+ dP &= 0.\\
d^{\ast}U^{\ast}&=0
\end{split}
\end{equation}

Now, let $\phi \in C^{0}(S(\infty))$, then by Theorem \ref{thmharmonic} there exists a (unique) harmonic function $F\in C^{\infty}(\H)\cap C^{0}(\bar{\H})$ satisfying $F\big|_{S(\infty)}=\phi$.  We let our initial data $u_{0}=dF$, and define a solution $(U^{\ast},P)$ to be 
\be\label{sol1}
U^{\ast}=\psi(t)dF\quad P=-\partial_{t}\psi(t)F-\frac 12 \abs{dF}^{2}+2a^{2}F
\ee
where $\psi$ is any function satisfying
\be\label{psi}
\psi^{2}(t) + 2{a^{2}}\int_{0}^{t} \psi^{2}(s) ds \leq \psi^{2}(0).
\ee
For example, we could let $\psi(t) = \exp(-\frac{At}{2})$ for some $A \geq 2a^{2}$.
\newline\indent
First we show $(U^{\ast},P)$ solves $\NS\H$.  This is very easy by the preparations we have done in Section \ref{hodge}.  Indeed, by \eqref{simple} $\Delta U^{\ast}=0$, and we observe that by Lemma \ref{convection} and \eqref{ricci2}, $$\partial_{t}U^{*}  + \overline{\nabla}_{U}U^{*} + 2\Ric(U^{\ast})=-dP.$$
\indent
It is also very easy to see that $U^{\ast}$ is divergence free since  by definition of $U^{\ast}$ and \eqref{DF} this is equivalent to $F$ being harmonic.
\newline\noindent
{\textbf {Global energy inequality:} }
Recall we want to show
\[
\int_{\H}\abs{U^{\ast}(t,x)}^{2}+2\int^{t}_{0}\int_{\H}\overline{g}(  \Def U^{\ast},\Def U^{\ast}  ) (s,x)ds\leq \int_{\H}\abs{u_{0}}^{2}.
\]
Thanks to Lemma \ref{Deflessnabla} we have
\begin{align*}
&\int_{\H}\abs{U^{\ast}(t,x)}^{2}+2\int^{t}_{0}\int_{\H}\overline{g}(  \Def U^{\ast},\Def U^{\ast}  ) (s,x)ds\\
&\qquad\qquad\leq \int_{\H}\abs{U^{\ast}(t,x)}^{2}+2\int^{t}_{0}\int_{\H} \overline{g}(  \overline{\nabla} U^{\ast}  , \overline{\nabla}U^{\ast}  ) (s,x) ds
\end{align*}
By Corollary \ref{finitedissipationonH}
\begin{align*}
\int_{\H} \overline{g}(\overline{\nabla}(dF) , \overline{\nabla}(dF)) = a^{2}\int_{\H}\abs{\nabla F(t,x)}^{2} ,
\end{align*}
so
\begin{align*}
\int_{\H}  \overline{g}(  \overline{\nabla} U^{\ast}  , \overline{\nabla}U^{\ast}  )    =a^{2}\int_{\H}{\psi}^{2}(t)\abs{\nabla F(t,x)}^{2}.
\end{align*}
Hence by \eqref{samenorm} and \eqref{psi}
\begin{align*}
&\int_{\H}\abs{U^{\ast}(t,x)}^{2}+2\int^{t}_{0}\int_{\H} \overline{g}(  \overline{\nabla} U^{\ast}  , \overline{\nabla}U^{\ast}  )ds\\
&\qquad\qquad=\int_{\H}\psi^{2}(t)\abs{\nabla F (x)}^{2}+ 2{a^{2}}\int^{t}_{0}\int_{\H}{\psi}^{2}(s)\abs{\nabla F (x)}^{2}ds\\
&\qquad\qquad=\left(\psi^{2}(t)+ 2{a^{2}}\int^{t}_{0}\psi^{2}(s)ds\right)\int_{\H}\abs{\nabla F (x)}^{2}\\
&\qquad\qquad\leq \psi^{2}(0)\int_{\H}\abs{\nabla F (x)}^{2}\\
&\qquad\qquad= \int_{\H}\abs{u_{0}}^{2},
\end{align*}
as needed.
\subsection{Proof of Theorem \ref{thm2}}
The proof is very similar.  Therefore we just give a brief sketch.\\
{\textbf {Existence and Non-uniqueness:} }
Again, for convenience, we recall the {\emph{modified}} Navier-Stokes equation on $M$.
\begin{equation}\tag{$\NS{M}$}
\begin{split}
\partial_{t}U^{*} - \Delta U^{*}  + \overline{\nabla}_{U}U^{*} + dP &= 0,\\
d^{\ast}U^{\ast}&=0.
\end{split}
\end{equation}
Let $u_{0}=dF$, and
\be\label{sol2}
U^{\ast}=\psi(t)dF\quad P=-\partial_{t}\psi(t)F-\frac 12 \abs{dF}^{2} 
\ee
with $\psi(t) = \exp(-\frac{At}{2})$ for some $A > 2b^{2}$..   Then as before we can see the equation is satisfied.
\newline\noindent
{\textbf {Global energy inequality:} }
 By \ref{Deflessnabla} we need to establish
\[
\int_{M}\abs{U^{\ast}(t,x)}^{2}+2\int^{t}_{0}\int_{M}  \overline{g}(  \overline{\nabla} U^{\ast}  , \overline{\nabla}U^{\ast}  )ds \leq \int_{M}\abs{u_{0}}^{2}.
\]
By Proposition \ref{finitedissipation}
\begin{align*}
\int_{M}  \overline{g}(  \overline{\nabla} (dF)  , \overline{\nabla}(dF)  ) \leq b^{2}\int_{M}\abs{\nabla F}^{2} ,
\end{align*}
so
\begin{align*}
\int_{M}   \overline{g}(  \overline{\nabla} U^{\ast}  , \overline{\nabla}U^{\ast}  ) \leq b^{2}\int_{M}{\psi}^{2}(t)\abs{\nabla F(t,x)}^{2} ,
\end{align*}
Hence  
\begin{align*}
&\int_{M}\abs{U^{\ast}(t,x)}^{2}+2\int^{t}_{0}\int_{M}  \overline{g}(  \overline{\nabla} U^{\ast}  , \overline{\nabla}U^{\ast}  )  ds\\
&\qquad\qquad\leq\int_{M}\psi^{2}(t)\abs{\nabla F (x)}^{2}+ 2{b^{2}}\int^{t}_{0}\int_{M}\abs{\psi(s)\nabla F (x)}^{2}ds\\
&\qquad\qquad=\left(\psi^{2}(t)+ 2{b^{2}}\int^{t}_{0}\psi^{2}(s)ds\right)\int_{M}\abs{\nabla F (x)}^{2}\\
&\qquad\qquad\leq \psi^{2}(0)\int_{M}\abs{\nabla F (x)}^{2}\\
&\qquad\qquad= \int_{M}\abs{u_{0}}^{2},
\end{align*}
as needed.

\subsection{Proof of Corollary \ref{cor1} and \ref{cor2}}\label{corproofs}
In the recent paper \cite{KNSS}, Koch, Nadirashvili, Seregin and \v{S}ver\'ak studied Liouville thorems and their consequences for the Navier-Stokes equations.  One of the results is
\begin{nonumthm}\cite{KNSS}
Let $u$ be a bounded weak solution of the Navier-Stokes equations on $\R^{2}\times(-\infty,0)$.  Then $u(x,t)=b(t)$ for a suitable bounded measurable function $b:(\infty,0)\rightarrow \R^{2}$.
\end{nonumthm}
The three dimensional problem is more difficult.  Nevertheless Koch, Nadirashvili, Seregin and \v{S}ver\'ak  are able to obtain corresponding results for the axi-symmetric equations with no swirl.   What Corollaries \ref{cor1} and \ref{cor2} show that in the hyperbolic setting we can have bounded solutions in both two and three dimensions (in fact, for any $n\geq2$) that are not functions of time only.  The nontrivial bounded solutions we choose are in the form of \eqref{sol1} and \eqref{sol2} for $\NS{H^{n}(-a^{2})}$ and \eqref{MNS} respectively, where we drop the condition \eqref{psi}, which is only needed to show the global energy inequality.  It would be interesting to find out whether or not these are the \emph{only} bounded solutions of $\NS{H^{n}(-a^{2})}$ and \eqref{MNS} .
\newline\indent
Here we also mention the result of Galdi \cite{Galdi}, which states
\begin{nonumthm}\cite{Galdi}
For the steady Navier-Stokes equation on $\R^{3}$ whenever the solution satisfies the finite dissipation property and $u\in L^{9/2}_{x}$, then $u$ must be a trivial solution, i.e. $u$ is constant.  
\end{nonumthm}
We note that in our case, we have a nontrivial solution, which belongs to $L^{9/ 2}_{x}$, but at this time we cannot say whether or not there exist nontrivial solutions in three dimensions that also satisfy the finite dissipation property.
\newline\noindent
{\textbf{Proof of Corollary \ref{cor1} and \ref{cor2}}}
Let $\psi$ be bounded in  \eqref{sol1} and \eqref{sol2} and without the condition \eqref{psi}.  It is obvious from the proof of Theorems \ref{mainthm} and \ref{thm2} that  \eqref{sol1} and \eqref{sol2} are the solutions of $\NS{H^{n}(-a^{2})}$ and \eqref{MNS} respectively since this is independent of the dimension of the underlying manifold.  Hence, we only need to verify that $U^{\ast}$ is $L^{\infty}$ bounded.  That can be checked in more than one way as follows (note we also do not need the exponential decay of the gradient of $F$).  
 Since $F \in C^{\infty}(M) \cap C^{0}(\overline{M})$, $\abs{dF}\leq \infty$ or we could use the much more sophisticated tool of the gradient estimate, Theorem \ref{thmgradient}.
\bibliography{ChiHinMagda}
\bibliographystyle{plain}

\end{document}